\documentclass[11pt]{article}

\usepackage{latexsym}
\usepackage{amssymb}
\usepackage{amsthm}
\usepackage{amscd}
\usepackage{amsmath}
\usepackage{verbatim}
\usepackage{xcolor}
\usepackage{enumerate}
\usepackage{graphicx}
\usepackage[ngerman,french,english]{babel}
 \usepackage[T1]{fontenc}

\graphicspath{ {KP-cone project/}}

\newtheorem{theorem}{Theorem}[section]

\newtheorem{claim}{Claim}

\newtheorem{lemma}[theorem]{Lemma}
\newtheorem*{lemma*}{Lemma}

\newtheorem*{definition*}{Definition}
\newtheorem{definition}[theorem]{Definition}
\newtheorem{proposition}[theorem]{Proposition}
\newtheorem{corollary}[theorem]{Corollary}

\theoremstyle{definition}
\newtheorem*{example}{Example}

\theoremstyle{remark}
\newtheorem{remark}{Remark}[section]

\numberwithin{equation}{section}

\let\inf\relax \DeclareMathOperator*\inf{\vphantom{p}inf}

\newcommand{\leqn}{\begin{equation}\label}
\def\endeqn{\end{equation}}

      \def\RR{\mathbb{R}}

\newcommand{\NN}{\mathbb{N}}

\def\dim{\mathrm{dim}}

\newcommand{\supp}{\mathrm{supp\,}}

\newcommand{\res}{\hbox{{\vrule height .22cm}{\leaders\hrule\hskip .2cm}}}

\makeatletter
\renewcommand{\@makefnmark}{\mbox{\textsuperscript{}}}
\makeatother

\def\adots{\mathinner{\mkern2mu\raise0pt\hbox{.}  
\mkern2mu\raise4pt\hbox{.}\mkern1mu
\raise7pt\vbox{\kern7pt\hbox{.}}\mkern1mu}}
\def\res{\hbox{ {\vrule height .22cm}{\leaders\hrule\hskip.2cm} } }

\allowdisplaybreaks[1]

\begin{document}
\title{Conical 3-uniform  measures: a family of new examples and characterizations }
\author{ A. Dali Nimer}
\date{}
\maketitle\footnote{The author was partially supported by NSF RTG 0838212, DMS-1361823 and DMS-0856687}
\footnote{Department of Mathematics, University of Chicago, 5734, S. University Ave., Chicago, IL, 60637

E-mail address: nimer@uchicago.edu}
\footnote{Mathematics Subject Classification Primary 28A33, 49Q15}

\selectlanguage{english}
\begin{abstract}
Uniform measures have played a fundamental role in geometric measure theory since they naturally appear as tangent objects. For instance, they were essential in the groundbreaking work of Preiss on the rectifiability of Radon measures. However, relatively little is understood about the structure of general uniform measures. Indeed, the question of whether there exist any non-flat uniform measures beside the one supported on the light cone has been open for 30 years, ever since Kowalski and Preiss classified $n$-uniform measures in $\mathbb{R}^{n+1}$ .

In this paper, we answer the question and construct an infinite family of 3-uniform measures in arbitrary codimension. We define a notion of distance symmetry for points and prove that every collection of $2$-spheres whose centers are distance symmetric gives rise to a $3$-uniform measure. We then develop a combinatorial method to systematically produce distance symmetric points.
We also classify conical $3$-uniform measures in $\mathbb{R}^{5}$ by proving that they all arise from distance symmetric spheres.
\end{abstract}

\section{Introduction}

The study of uniform measures was an essential part of Preiss's  proof of the rectifiability of measures in Euclidean space.
They have played a fundamental role in geometric measure theory ever since as they naturally appear as tangent measures. 
In layman's terms, a tangent measure at a point is seen by zooming in on the measure near that point.
At almost every point of positive and finite $n$-density in the support of a Radon measure, the tangent measures are $n$-uniform.
A geometric understanding of $n$-uniform measures is thus crucial in describing the infinitesimal and asymptotic geometry of a large class of measures.

 Relatively little was known about $n$-uniform measures. Indeed the only example of a non-flat uniform measure is due to Preiss in $\cite{P}$. It is given by $\mathcal{H}^{3} \res C$ where $C$ is the light cone described by
 \begin{equation} \label{KPcone}
C= \left\lbrace x \in \mathbb{R}^{4} \; ; \; {x_4}^2={x_1}^2+{x_2}^2+{x_3}^2 \right\rbrace.
\end{equation}
 The question of the existence of other uniform measures has been open since Kowalski and Preiss proved a classification result for  $n$-uniform measures in $\mathbb{R}^{n+1}$ in 1987.
 
 In this paper, we answer the question and construct an infinite family of 3-uniform measures in arbitrary codimension.
 Moreover, we classify conical $3$-uniform measures in $\mathbb{R}^{5}$.
We also provide a description of the structure of conical $3$-uniform measures and develop a combinatorial method to systematically produce new examples.
 
We start by introducing some definitions in order to give precise statements of our results.
We say a Radon measure $\mu$ in $\mathbb{R}^d$ is uniformly distributed if there exists a real-valued function $\phi$ so that for every $x \in \supp(\mu)$, and every $r>0$ $$\mu(B(x,r))=\phi(r).$$
If there exists $c>0$ so that 
\begin{equation} \label{nunif} \phi(r)=cr^{n}, 
\end{equation}
we call $\mu$ an $n$-uniform measure. More generally, we will say $\mu$ is support $n$-uniform if it satisfies $\eqref{nunif}$ for $0 \leq r \leq D$, where $D$ is the diameter of $\supp(\mu)$. 
Some obvious examples of $n$-uniform measures  are $n$-flat measures, i.e. $n$-Hausdorff measure restricted to an affine $n$-plane. Indeed, if $V$ is an affine $n$-plane then for all $x \in V$ and $r>0$, we have:
$$\mathcal{H}^{n}(B(x,r) \cap V) =\omega_{n} r^{n},$$
where $\omega_{n}$ denotes the volume of the $n$-dimensional unit ball.
In fact, Preiss proved in $\cite{P}$ that for $n=1,2$, the only $n$-uniform measures in $\mathbb{R}^{d}$ are the $n$-flat ones.

In $\cite{P}$, Preiss discovered an example of a non-flat $n$-uniform measure and proved in collaboration with Kowalski (see $\cite{KoP}$) that in codimension 1, this measure and flat measures are  the only examples of $n$-uniform measures.

\begin{theorem}[\cite{KoP}]\label{classificationKP}
Let $C$ be the cone in $\mathbb{R}^{4}$ defined by: 
\begin{equation}
C= \left\lbrace (x_1, x_2,x_3,x_4) \in \mathbb{R}^{4} ; {x_4}^2={x_1}^2+{x_2}^2+{x_3}^2 \right\rbrace.
\end{equation}
Then :
\begin{itemize}
\item $\mathcal{H}^{3} \res C$ is $3$-uniform and for all $x \in C$, for all $r>0$,
 \begin{equation}\label{KPcone2}
\mathcal{H}^{3} (B(x,r)\cap C)=\frac{4}{3}\pi r^{3}.
\end{equation}
\item If $\mu$ is an $n$-uniform measure in $\mathbb{R}^{n+1}$, then either $\mu$ is $n$-flat or, up to isometry, we have:
\begin{equation}
\mu= c \mathcal{H}^{n} \res \left( C \times \mathbb{R}^{n-3} \right).
\end{equation}
\end{itemize}
\end{theorem}

In higher codimension, there is no such classification result. However, in $\cite{KiP}$, Kirchheim and Preiss proved that the support of an $n$-uniform measure in any codimension is an analytic variety.
 
 \begin{theorem}\label{uniformlydistmeasure}[\cite{KiP}]
Let $\mu$ be a uniformly distributed measure over $\mathbb{R}^{d}$. Then $\supp(\mu)$ is a real analytic variety and there exists an integer $n \in \left\lbrace0,1,\ldots,d \right\rbrace$, a constant $c \in (0,\infty)$ and an open subset $G$ of $\mathbb{R}^{d}$ such that:
\begin{enumerate}
\item $G \cap \supp(\mu)$ is an $n$-dimensional analytic submanifold of $\mathbb{R}^{d}$;
\item ${\mathbb{R}^{d}} \backslash G$ is the union of countably many analytic submanifolds of $\mathbb{R}^{d}$ of dimensions less than $n$ and $\mu(\mathbb{R}^{d} \backslash G)= \mathcal{H}^{n}(\mathbb{R}^{d} \backslash G)=0$;
\item $\mu(A)= c \mathcal{H}^{n}(A \cap G \cap \supp(\mu))= c \mathcal{H}^{n}(A \cap \supp(\mu))$ for every subset $A \subset \mathbb{R}^{d}$.
\end{enumerate}
We denote $G \cap \supp(\mu) $ by $\mathcal{R}$ and $\supp(\mu) \backslash G $ by $\mathcal{S}$ and write:
$$ \supp(\mu) = \mathcal{R} \cup \mathcal{S}.$$
\end{theorem}

One of the central aspects of the following paper is to produce new examples of $3$-uniform measures. 

The first insight behind these constructions is identifying Archimedes's theorem as the reason that the light cone supports a $3$-uniform measure. 
A conical $3$-uniform measure can be viewed as a cone over a support $2$-uniform measure (See Theorem $\ref{3unif}$). And Archimedes's theorem states that a $2$-sphere is support $2$-uniform. This suggests that the key to finding $3$-uniform measures is to take appropriate unions of $2$-spheres. In the case of the light cone, the intersection with $\mathbb{S}^{3}$ gives $2$ $2$-spheres. These $2$-spheres are in the exact position that forces their union to be support $2$-uniform  (the fact that they are locally $2$-uniform is a consequence of Archimedes).

The second insight consists in isolating the condition of distance symmetry as being the one that makes this union of $2$-spheres support $2$-uniform (see Definition $\ref{layers}$). This means that from every center of a sphere, the set of distances to the other centers is the same. It allows us to reduce the problem of constructing a support $2$-uniform measure supported on a sphere (a fortiori a $3$-uniform measure) to the combinatorial one of producing points in Euclidean space with a given distance set. In particular, we construct a family of $3$-uniform measures in arbitrary codimension.
\begin{theorem}\label{Ck}
For every $k=0,1, \ldots$, let $C_k$ be the  cone in $\mathbb{R}^{k+4}$ consisting of the points $x=(x_1, \ldots, x_{k+4})$ satisfying
\begin{equation*}
\left\lbrace x \in \RR^{d} \; ; \;x_4^2= x_1^2 + x_2^2+ x_3^2 \right\rbrace \cap \bigcap_{l=1}^{k} \left\lbrace x \in \RR^{d} \; ; \;x_{l+4}^{2} = 2^{l} x_4^{2} \right\rbrace.
\end{equation*}
Then, for all $x \in C_k$, for all $r>0$
\begin{equation*}
\mathcal{H}^{3}(B(x,r) \cap C_k )= \frac{4}{3} \pi r^{3}.
\end{equation*}
\end{theorem}
 It turns out that in codimension $2$, this family gives all possible non-flat $3$-uniform measures with dilation invariant support leading to the following classification result.
\begin{theorem} \label{charactergeom}
Let $\nu$ be a conical Radon measure in $\RR^{5}$ (i.e for all $r>0$, $\supp(\nu)=r \supp(\nu)$) and let $\Sigma:=\supp(\nu)$. Then $\nu$ is a $3$-uniform measure if and only if there exists $c>0$ such that, up to isometry, \begin{equation}  \nu= c \mathcal{H}^{3} \res \Sigma ,
\end{equation}
where $\Sigma$ is one of the three following sets 
\begin{enumerate}
\item an affine $3$-plane $V$,
\item $\left\lbrace x \in \RR^{5} \; ; \; x_4^2= x_1^2 + x_2^2+ x_3^2 \right\rbrace \cap \left\lbrace x \; ; \; x_5=0\right\rbrace$, or
\item $\left\lbrace x \in \RR^5 \; ; \; x_4^2= x_1^2 + x_2^2+ x_3^2 \right\rbrace \cap \ \left\lbrace x \; ; \; x_{5}^{2} = 2 x_4^{2} \right\rbrace.$
\end{enumerate}

\end{theorem}

We describe the structure of the paper in more detail. Our first step is to obtain a description of the spherical component $\sigma$ of a conical $3$-uniform measure $\nu$.
Theorem $\ref{uniformlydistmeasure}$ says that almost every point of the support of an $n$-uniform measure is smooth.  
  With this in mind, in $\cite{KoP}$, Kowalski and Preiss start by considering a locally $n$-uniform measure with smooth support $M$. Fixing a point $x$ in its support and using the area formula, they write a Taylor expansion for the measure of $B(x,r)$, in terms of $r$.
  By equating this expansion with $\omega_{n} r^{n}$, they prove that in the case where $n=2$, the ambient space is $\RR^{3}$, and  the manifold $M$ is connected, $M$ has to be a piece of a $2$-plane or of a $2$-sphere.
  In Section 3,  we carry out a similar argument on $\sigma$, the spherical component of $\nu$, where the ambient space is $\RR^{d}$, $d >3$,  to  deduce that it is an umbilic manifold. As the proof is analogous to the argument in $\cite{KoP}$, we include it as an appendix for the reader's convenience. 
  Using the fact that the measure is support $2$-uniform and that its support is an analytic variety, we prove that its support is in fact a finite union of disjoint $2$-spheres (see Theorem $\ref{finitelymanyspheres}$).
We then study the configuration of these spheres. Indeed, the fact that $\sigma$ is support $2$-uniform implies a certain rigidity. In Theorem $\ref{suffcond}$, we find a sufficient condition for a configuration of $2$-spheres in $\RR^{d}$ to be the support of a  support $2$-uniform measure. They must have the same radius and be contained in translations of the same linear $3$-plane.
   Moreover, their centers have to be in a specific position: we say they are distance symmetric (see Definition $\ref{layers}$).
 In Theorem $\ref{necessarycond}$, we show that when $d=5$, this condition is in fact necessary, thus giving a classification of $3$-uniform conical measures.


In Section 4, we explicitly construct an infinite family of non-isometric $3$-uniform measures in Euclidean spaces of different dimensions. To do that, we first construct rectangular parallelotopes whose vertices are distance symmetric (see Lemma $\ref{parallelotope1}$).
Using this construction, we produce a family of $3$-uniform measures in arbitrary codimension.

In Section 5, we use combinatorial methods to systematically produce all distance symmetric points. We construct a graph associated to  a configuration of distance symmetric points and in Lemma $\ref{spectralgap}$, we translate the existence of such a configuration in Euclidean space to a necessary and sufficient condition on the graph. The advantage of this condition is that it is computable, expressed as a bound on the eigenvalues of the Laplacian matrix associated to the graph.
We finally prove Theorem $\ref{constructpoints}$ where we describe how to find the coordinates of those centers in the corresponding ambient space and the rank of the linear space generated by the centers. This method allows us to produce examples that are less symmetric than the ones constructed in Section 5. To illustrate this, we construct one such example.
\subsection{Acknowledgements}
The author would like to thank Professor T. Toro for her support, guidance and feedback. We would also like to thank Professors D. Preiss and M. Badger for their helpful comments on a preliminary version of this paper.

\section{Preliminaries}
\subsection{Geometry and Analysis}
Let $\mu$ be a measure in $\RR^{d}$. We define the support of $\mu$ to be 
\begin{equation}
\supp(\mu)=\left\lbrace x \in \RR^{d} ; \mu(B(x,r))>0, \mbox{ for all } r>0 \right\rbrace. 
\end{equation}
Note that the support of a measure is a closed subset of $\RR^{d}$.

\begin{definition}
Let $\mu$ be a Radon measure in $\RR^d$ and denote its support by $\Sigma$.
\begin{itemize}
\item We say $\mu$ is uniformly distributed if there exists a positive function $\phi : \RR_{+} \rightarrow \RR_{+}$ such that:
$$
\mu(B(x,r))=\phi(r), \text{ for all } x \in \Sigma, r>0.$$
We call $\phi$ the distribution function of $\mu$.

\item If there exists $c>0$ such that $\phi(r)=c r^n$, we say $\mu$ is $n$-uniform.

\item If $\mu$ is an $n$-uniform measure such that $T_{0,r}[\mu] = r^{n} \mu $ for all $r>0$, we call it a conical $n$-uniform measure, where $T_{0,r}[\mu]$ is the  push-forward of $\mu$ by the dilation $$T_{0,r}(y)=\frac{y}{r}.$$

\end{itemize}

\end{definition}
We will use this result throughout the paper: it  says that for an $n$-uniform measure, the support and the measure can be essentially identified.

\begin{theorem} [\cite{KoP}] \label{support}
Let $\mu$ be an $n$-uniform measure in $\RR^{d}$ with $\Sigma=\mbox{supp}(\mu)$ and let $c>0$ be such that for $x \in \Sigma$, $r>0$ \begin{equation}
\mu(B(x,r))=cr^{n}.
\end{equation} 
Then  $\Sigma$ is $n$-rectifiable and 
\begin{equation}\label{supportunifmeasure}
\mu = c  \omega_{n}^{-1}\mathcal{H}^{n} \res \Sigma.
\end{equation}
\end{theorem}

We state the area and the coarea formulae which will be used in this paper.

\begin{theorem}[\cite{Si}][The area formula]
Let $f: \RR^{m} \rightarrow \RR^{d}$ be a 1-1 $C^{1}$ function where $m<d$. Then, for any  Borel set $A \subset \RR^{m} $,we have:
\begin{equation}\label{area}
\int_{A} Jf(x) d\mathcal{L}^{m}(x)= \mathcal{H}^{m}(f(A))
\end{equation}
where \begin{equation}
Jf(x)= \sqrt{det((df(x))^{*}  \circ df(x)) },
\end{equation}
and $(df(x))^{*}$ is the adjoint of $df(x)$.
\end{theorem}

\begin{theorem}[\cite{Si}][The co-area formula]
Let $M \subset \RR^d$ be an $n$-rectifiable set and $f: M \rightarrow \RR^{m}$, $m < n \leq d$  a Lipschitz function. Then for any non-negative Borel function $g:M \rightarrow \RR$, we have:
\begin{equation}\label{coarea}
\int_{M} g(x) J_{M}^{*}f (x)d\mathcal{H}^{n} (x) = \int_{\RR^{m}} \int_{f^{-1}(y) \cap M} g(z) d\mathcal{H}^{n-m}(z) d\mathcal{L}^{m}(y),
\end{equation}
where 
\begin{equation}
 J_{M}^{*}f(x)= \sqrt{det(d^{M}f(x) \circ (d^{M}f(x))^{*} ) }.
\end{equation}
\end{theorem}

We now state two theorems which will be crucial to the description of the geometry of the spherical components.
In $\cite{KoP}$, Kowalski and Preiss proved that the curvature of a manifold whose surface measure is locally $n$-uniform must satisfy the following equation.
 \begin{theorem} \label{KoP}[\cite{KoP}]
  If a hypersurface $M \subset \mathbb{R}^{n+1}$ of class $C^{5}$ is such that for all $x \in M$, there exists $r_0>0$ such that for all $r<r_0$,
  \begin{equation}
  \mathcal{H}^{n}(B(x,r) \cap M) = \omega_{n} r^{n},
  \end{equation}
  then we have along $M$: 
$$h^{2} = 2 || \overrightarrow{h} ||^{2}=2 \tau,$$
 
  where $\overrightarrow{h}$ denotes the second fundamental form, $h$ the trace of $\overrightarrow{h}$, $\tau$ the scalar curvature and $|| . ||$ the norm of a tensor with respect to the Riemannian inner product.
\end{theorem}   
When $n=2$, this theorem essentially says that all points of the manifold are umbilic. 
The following is a classical geometry theorem describing umbilic manifolds. 
\begin{theorem}\label{umbilicmanifold }[\cite{Sp}, Chapter 7 Theorem 29]
For $n \geq 2$, let $M^{n} \subset \mathbb{R}^{d}$ be a connected immersed submanifold of $\mathbb{R}^{d}$ with all points umbilics. Then either $M$ lies in some $n$-dimensional plane or else $M$ lies in some $n$-dimensional sphere in some $(n+1)$-dimensional plane.
\end{theorem}

In $\cite{KiP}$, Kirchheim and Preiss proved that the support of a uniformly distributed measure is an analytic variety.
We need the following theorem by Lojasiewicz to describe the geometry of an analytic variety.
\begin{theorem}\label{Lojastructtheorem} [\cite{L}]
Let $\Phi(x_1, \ldots, x_d)$ be a real analytic function on $\mathbb{R}^{d}$ in a neighborhood of the origin. We may assume $\Phi(0, \ldots,0, x_d) \neq 0$. After a rotation of the coordinates $(x_1, \ldots, x_{d-1})$, one has that there exist numbers $\delta_{j}>0$, $j=1, \ldots, d$ such that the set $Z$ defined as :
$$Z= \left\lbrace x = (x_1, \ldots, x_{d}) : \; |x_j| < \delta_j , \mbox{ for all } j  \mbox{ and } \Phi(x)=0 \right\rbrace,$$
has a decomposition
\begin{equation}\label{Lojastrat}
Z= V^{d-1} \cup \ldots \cup V^{0}.
\end{equation}
The set $V^{0}$ is either empty or consists of the origin alone. For $1 \leq k \leq d-1$, we may write $V^{k}$ as a finite, disjoint union of analytic $k$-submanifolds of $\mathbb{R}^{d}$.

Moreover, $Z$ is stratified in the following sense: for each $k$, the closure of $V^{k}$ contains all the subsequent $V_j$'s, i.e. defining $Q$ to be 
$$Q=\left\lbrace x \in \mathbb{R}^{d} ; \; |x_j|<\delta_j, \mbox{ for all } j \right\rbrace,$$
we have:
\begin{equation}\label{Lojastrat2}
V^{0} \cup \ldots \cup V^{k-1} \subset \overline{Q \cap V^{k}}.
\end{equation}
\end{theorem}

The following results about conical $n$-uniform measures will also be essential in the proofs of the main results.
We start with a definition.

\begin{definition}
Let $\nu$ be a conical $n$-uniform measure in $\mathbb{R}^d$, with $0$ in its support, $\Sigma$ its support.
We define $\sigma$ the spherical component of $\nu$, to be:
$$\sigma = \mathcal{H}^{n-1} \res ({\Sigma \cap {S^{d-1}}}),$$
where $S^{d-1}= \left\lbrace x \in \RR^{d} ; |x|=1 \right\rbrace$.
\end{definition}
We have a polar decomposition for conical $n$-uniform measures.

\begin{lemma}\label{polar} [\cite{N}] Let $\nu$ be a conical $n$-uniform measure in $\RR^{d}$. Let $g$ be a Borel function on $\RR^{d}$. Then:
\begin{equation}\label{eqpolar}
\int g(x) d\nu(x)= \int_{0}^{\infty} \rho^{n-1} \int g(\rho x') d\sigma(x') d\rho, 
\end{equation} 
where $\rho=|x|$ and $x'=\frac{x}{|x|}$. \end{lemma}

The following results state that the spherical component of a conical $n$-uniform measure is uniformly distributed and give an expression for its distribution function $\phi$ when $n=3$.

  \begin{theorem}\label{spherunifdist} [\cite{N}] Let $\nu$ be a conical $n$-uniform measure in $\mathbb{R}^d$. Then $\sigma$ the spherical component of $\nu$  is a uniformly distributed measure. \end{theorem}

  \begin{corollary}\label{3unif}[\cite{N}] Suppose $\nu$ a $3$-uniform conical measure on $\RR^d$. Let $\sigma$ be its spherical component, and denote the support of $\sigma$ by $\Omega$. Then there exists a function $\phi: \RR_{+} \rightarrow \RR_{+} $ such that, for all $x \in  \Omega$, for all $r>0$:
  \begin{equation} \sigma(B(x,r))=\phi(r). 
  \end{equation} 
Moreover,   
  \begin{equation}
  \phi(r)=\pi r^{2} \chi_{(0,2)}(r)+4 \pi \chi_{2,\infty}(r). 
  \end{equation}
   \end{corollary}
  
 The following corollaries are two consequences of Corollary $\ref{3unif}$.
  \begin{corollary} \label{algvarcone}[ \cite{N}]
 Let $\nu$ be a conical $n$-uniform measure in $\mathbb{R}^d$ and $\Sigma$ its support. Then $\Sigma$ is an algebraic variety and \begin{equation}
 \Sigma=-\Sigma.
 \end{equation}
 \end{corollary}

  Corollary $\ref{3unif}$ says that the spherical component of a conical $3$-uniform measure is  support $2$-uniform. The following proves the converse: if $\Omega$ is a subset of $\mathbb{S}^{d-1}$ such that$\mathcal{H}^{2} \res \Omega$ is  support $2$-uniform, and $\Sigma$ is the cone over $\Omega$ then $\mathcal{H}^{3} \res \Sigma$ is $3$-uniform.

\begin{lemma}\label{cond3unifconical}
Let $\Omega$ be a set in $\RR^{d}$ contained in $\mathbb{S}^{d-1}$, $\sigma= \mathcal{H}^{2} \res \Omega$ and assume that $\sigma$ satisfies the property that for all $x \in \Omega$, for $r \leq 2$,
\begin{equation}
 \sigma(B(x,r))=\pi r^{2}.
\end{equation}

Define $\Sigma$ to be:
\begin{equation}
\Sigma=\left\lbrace x \in \RR^{d} ; \frac{x}{|x|} \in \Omega \right\rbrace \cup \left\lbrace 0 \right\rbrace,
\end{equation}
and $\nu$ to be $\mathcal{H}^{3} \res \Sigma$.

Then 
for all $x\in \Sigma$, for $r>0$, we have:
\begin{equation}
\nu(B(x,r))=\frac{4}{3}\pi r^{3}.
\end{equation}
In particular, $\nu$ is $3$-uniform.
\end{lemma}
\begin{proof}
We prove that $\nu(B(e,r))=\frac{4}{3} \pi r^{3}$, for $e \in \Omega$, $r>0$.
The theorem then follows for any $x \in \Sigma$.
Indeed, if $x \in \Sigma$, $x \neq 0$ then $e=\frac{x}{|x|} \in \Omega$. Moreover, by the definition of $\Sigma$ we have $\frac{\Sigma}{u}= \Sigma$ for any $u>0$.
This gives:
$$
\mathcal{H}^{3}(B(x,r) \cap \Sigma) = \mathcal{H}^{3}\left(|x| \left(B\left(e, \frac{r}{|x|}\right) \cap\frac{\Sigma}{|x|}\right)\right)= |x|^{3}\mathcal{H}^3\left(B\left(e,\frac{r}{|x|}\right) \cap \Sigma\right)= \frac{4}{3}\pi r^{3}.
$$
On the other hand, let $x_i=\frac{e}{i}$ for some $e \in \Omega$ and let $r>0$. Then since $ \chi_{B(x_i,r)}(z) \to \chi_{B(0,r)}(z),$ for $\nu$-almost every $z$, we get:

$$\frac{4}{3}\pi r^{3}=\lim_{i \to \infty} \nu(B(x_i,r)) = \nu(B(0,r)).$$

Let us now prove the theorem for $e \in \Omega$. Let $r>0$.
Then, by the co-area formula, 
\begin{equation}\label{coareaball}
\nu(B(e,r)) = \int_{0}^{\infty} \mathcal{H}^{2}(B(e,r) \cap \partial B_{\rho} \cap \Sigma) d\rho 
\end{equation}
 where $B_{\rho}$ denotes $B(0,\rho)$.

Let us compute $\mathcal{H}^{2}(B(e,r) \cap \partial B_{\rho} \cap \Sigma)$.

We first express $B(e,r) \cap \partial B_{\rho} \cap \sigma$, whenever it is non-empty, as a ball centered on $\rho e$.
Let $z \in B(e,r) \cap \partial B_{\rho} \cap \Sigma$. Then an easy calculation gives
$$
 |z|=\rho, \left| z - e \right|^{2} \leq r^{2} \iff  |z|=\rho, |z-\rho e|^{2} \leq 2 \rho^{2} + \rho (r^{2}-1)-\rho^{3}=: R^{2}.
$$
Consequently, $\Sigma$ being dilation invariant and $\Sigma \cap \mathbb{S}^{d-1}$ being support $2$-uniform, we get:

\begin{align*}
{H}^{2}(B(e,r) \cap \partial B_{\rho} \cap \Sigma)&= \mathcal{H}^{2}(B(\rho e,R) \cap \partial B_{\rho} \cap \Sigma),\\
&= \rho^{2} \mathcal{H}^{2}(B( e,\frac{R}{\rho}) \cap \mathbb{S}^{d-1} \cap \Sigma), \\
&=\pi R^{2},\\
&=\pi (2 \rho^{2} + \rho (r^{2}-1)-\rho^{3}).
\end{align*}

We consider two cases: when $r \leq 1$ and $r \geq 1$.

If $r \leq 1$, $B(e,r) \cap \partial B_{\rho} \cap \Sigma= \emptyset$ unless $1-r \leq \rho \leq 1+r$ and
$$
\nu(B(e,r)) = \int_{1-r}^{1+r} \mathcal{H}^{2}(B(e,r) \cap \partial B_{\rho} \cap \Sigma) d\rho 
=  \int_{1-r}^{1+r} \pi (2 \rho^{2} + \rho (r^{2}-1)-\rho^{3}) d\rho=\frac{4}{3} \pi r^{3}.
$$

In the case where $r \geq 1$, notice that when $\rho \leq r-1$ , $\partial B_{\rho} \subset B(e,r)$, and when $\rho > r+1$,  $\partial B_{\rho} \cap B(e,r) =  \emptyset$.
Therefore, we can write:

\begin{align*}
\nu(B(e,r))&= \int_{0}^{r-1}\mathcal{H}^{2}(\partial B_{\rho}) d\rho + \int_{r-1}^{r+1}  \mathcal{H}^{2}(B(e,r) \cap \partial B_{\rho} \cap \Sigma) d\rho , \\ &= 4\pi \int_{0}^{r-1} \rho^{2} d\rho +\int_{r-1}^{r+1} \pi (2 \rho^{2} + \rho (r^{2}-1)-\rho^{3}) d\rho, \\
&= \frac{4}{3}\pi r^{3}.
\end{align*}

\end{proof}

  We also state a theorem due to Archimedes: it says that the surface measure of a $2$-sphere is the support of a support $2$-uniform measure. We provide a proof using the area formula.
  
  \begin{lemma}[Archimedes]\label{spherelocally2unif}
  Let $S$ be a sphere of radius $R$ in $\mathbb{R}^{3}$. 
  Then for all $u \in S$, for all $\rho \leq 2R$, we have:
  \begin{equation}\label{spherelocally2unifeq}
  \mathcal{H}^{2}(B(u,\rho) \cap S)= \pi {\rho}^{2}.
  \end{equation}
  \end{lemma}
 \begin{proof} 
 Without loss of generality, Hausdorff measure being invariant under isometries and under dilation up to appropriate normalization, we can assume that $S=\mathbb{S}^{2}$ and $u=(0,0,1)$ . 

We claim that for $e=(0,0,1)$ and $r \leq 2$, \begin{equation}\label{spherelocally2unifeq2}
\mathcal{H}^{2}(\mathbb{S}^{2} \cap B(e,r))=\pi r^{2}.
\end{equation}
                                                                                                                                                                        First, note that $\partial B(e,r) \cap S^{2} = \{(x,y,z) \in \RR^3; x^2+y^2+z^2=1, x^2+y^2+(z-1)^2=r^2 \}$. If $r<\sqrt{2}$, $B(e,r) \cap S^{2}$ is the portion of the graph of $f(x,y)=\sqrt{1-(x^2+y^2)}$ above $z=1-\frac{r^2}{2}$. So we have, by the area formula:
  
$$ \mathcal{H}^{2}(B(e,r) \cap S^{2}) = \int_{0}^{2\pi} \int_{0}^{\sqrt{1-(1-\frac{r^2}{2})^2} } \sqrt{1+|\nabla{f}|^2} \rho d\rho d\theta
                                                                        =\pi r^2. $$
                                                                                      
                                                                                      If $ \sqrt{2}<r <2$, $B(e,r)$ and $B(0,1)$ intersect in $z=1-\frac{{r}^2}{2}$.
                                                                                       Moreover, note that the part of $S^2$ below the plane $z=1-\frac{{r}^2}{2}$ is $B(-e,r')$, where, by applications of Pythagoras' theorem, we have $
{r'}^2 =4-r^{2}$ 
Therefore, by symmetry (since $r' < \sqrt{2}$), we have:
$$
\mathcal{H}^2(B(e,r) \cap {S^2}) =\mathcal{H}^2(S^2) - \mathcal{H}^2(B(-e,r')\cap S^2)=\pi r^{2}. 
$$
This proves $\eqref{spherelocally2unifeq2}$.

Therefore, since $\rho \leq 2R$, we have:

$$
\mathcal{H}^{2}\left(S \cap B(u,\rho)\right) = \mathcal{H}^{2}\left(R \left( \mathbb{S}^{2} \cap B\left(e, \frac{\rho}{R}\right) \right) \right) =R^{2} \pi \left(\frac{\rho}{R}\right)^{2}= \pi {\rho}^{2}.
$$
                                                                                                                                                                             
 \end{proof}
 
 \subsection{Discrete Mathematics}
 
In Section 4, we need to understand what conditions on a set of distances guarantees their embeddability in Euclidean space. To this end, we use a theorem of embeddability from $\cite{Bl}$.

\begin{definition}
Let $X$ be a set. We call $X$ a distance space if there exists a distance function $d_{X}: X \times X \to Y$, where $Y$ is called the distance set. Typically $Y$ will be taken to be $\mathbb{R}_{+}$.

We call a distance space $(X,d_X)$ semimetric if $d_{X}$ has co-domain $\mathbb{R}_{+} \cup \left\lbrace 0 \right\rbrace$ and if $d_X$ satisfies for all $p,q \in X$:
\begin{itemize}
\item $d_{X}(p,q)=0 \iff p=q$,
\item $d_{X}(p,q) =d_{X}(q,p)$.
\end{itemize}
\end{definition}

We remind the reader of the geodesic distance of two points on a sphere.

\begin{definition}
For two points $x,y \in t \mathbb{S}^{m} \subset \mathbb{R}^{m+1}$, for some $t>0$, we define the distance $| \; .\;  |_{t \mathbb{S}^{m}}$ to be: 
\begin{equation}\label{geodesic}
| x - y |_{t \mathbb{S}^{m}} = t . \arccos \left( \frac{\left\langle x,y \right\rangle}{t^{2}} \right),
\end{equation}
where $\left\langle  ,  \right\rangle$ is the Euclidean inner product.
\end{definition}

\begin{theorem}[\cite{Bl}] \label{Bl}
Let $X=\left\lbrace p_1, \ldots, p_n \right\rbrace$ be a semimetric space, $t>0$ and define the $n \times n$ matrix $\Delta$ to be:
\begin{equation}
\Delta= \left( \cos \left( \frac{d_{X}(p_i,p_j)}{t} \right) \right)_{i,j}.
\end{equation}

Then there exist points $\left\lbrace \xi_{i} \right\rbrace_{i=1}^{n}$ in $ t \mathbb{S}^{n-2}$ such that:
\begin{equation}
| \xi_{i} - \xi_{j} |_{t\mathbb{S}^{n-2}} = d_{X}(p_i,p_j)
\end{equation}
if and only if $d_{X}(p_i,p_j) \leq  \pi t $ and all of the matrix $\Delta$'s principal minors are non-negative (or equivalently $\Delta$ is positive semidefinite).
\end{theorem}

We give some basic notions of graph theory which will be used in the final section of this paper.

\begin{definition}
\begin{enumerate} 
Let $G$ be a graph. We denote the vertices of $G$ by $V(G)$, its edges by $E(G)$.

\item A weighted graph is a graph to which we associate a weight function $w: E(G) \to \mathbb{R}_{+}$.

\item The degree $d(v)$ of a vertex $v$ is defined as $d(v)= \sum_{u \sim v} w \left( \left\lbrace u,v \right\rbrace \right)$.
\item A $k$-edge coloring of $G$ is a function $c: E(G) \to \left\lbrace 1, \ldots, k \right\rbrace$ such that $c(e) \neq c(f)$ if $e$ is adjacent to $f$

\end{enumerate}
\end{definition}

\begin{example}
An example of a graph which will be used in Section $4$ is the complete graph $K_n$. This graph has $n$ vertices $V(G)=\left\lbrace v_{i} \right\rbrace_{i=1}^{n}$ and its edges are all the subsets of $V(G)$ of cardinality $2$ i.e. $E(G)= \left\lbrace \left\lbrace v_i v_j \right\rbrace \right\rbrace_{ 1 \leq i < j \leq n}$.

\end{example}

  To each graph are associated two matrices that encode information about its structure: the adjacency matrix and the Laplacian matrix.

  \begin{definition}
  Let $G$ be a weighted graph.
  
  \begin{enumerate}
  \item The adjacency matrix $A= (A_{ij})_{i,j}$ of $G$ is defined as:
  \begin{equation}
  A_{ij}= \begin{cases} 0, \mbox{ if } i=j, \mbox{ or } \left\lbrace v_i, v_j \right\rbrace \notin E(G) \\ w(\left\lbrace v_i, v_j \right\rbrace),\mbox{  if  } i \neq j, \left\lbrace v_i, v_j \right\rbrace \in E(G).
  \end{cases}
  \end{equation}
  
  \item The degree matrix $D$ of $G$ is the diagonal matrix with entries: \begin{equation}
  D_{ii}= d(v_i).
  \end{equation}
  \item The Laplacian $L= ( L_{ij})_{i,j}$  of $G$ is defined as \begin{equation}
  L= D - A,
  \end{equation}
where $D$ is the degree matrix.
Its second smallest eigenvalue $\lambda_{G}$ is called the spectral gap of $L$.   
 \end{enumerate}
  \end{definition}

\section{Description of the spherical component  of a conical $3$-uniform measure}
\subsection{The spherical component is a union of 2-spheres}
We now study the geometry of the support of the spherical component $\sigma$ of the $3$-uniform measure $\nu$.

Our first aim is to prove that $\Omega$ is a finite union of disjoint $2$-spheres.

\begin{theorem}\label{finitelymanyspheres}
Let $\nu$ be a conical $3$-uniform measure in $\RR^{d}$, $\sigma$ its spherical component and $\Omega$ the support of $\sigma$. Then 
\begin{equation}
\Omega = \bigcup_{i=1}^{M} S_{i},
\end{equation}
where the $S_i$'s are mutually disjoint $2$-spheres.
\end{theorem}

We start by stating the following intermediate lemma. Its proof is given as an appendix as it follows the proof of Theorem $\ref{KoP}$ very closely.

\begin{lemma}\label{umbilic}
Let $\mu$ be a $3$-uniform measure in $\RR^{d}$, $\sigma$ its spherical component and $\supp(\sigma)= \Omega$. Then:

$$ \mathcal{R} \subset \bigcup_{\alpha} S_{\alpha},$$
where the $S_{\alpha}$'s are $2$-spheres and $\mathcal{R}$ is the regular part of $\Omega$ as defined in Theorem $\ref{uniformlydistmeasure}$.
\end{lemma}

We now use Lemma $\ref{umbilic}$ to prove Theorem $\ref{finitelymanyspheres}$. 
\begin{proof}
Write 
\begin{equation} \label{analyticpieces}
\mathcal{R}= \cup_{i} M_i ,
\end{equation}
where each $M_i$ is an analytic  $2$-submanifold of $\RR^{d}$. Note that each $M_i$ might be disconnected (i.e a sphere $S_i$ might contain many disconnected ``pieces'' of $2$-spheres).

First, we claim that there are only finitely many $M_i$'s. Indeed, by Theorem $\ref{Lojastructtheorem}$, for every $x \in \Omega$, there exists a neighborhood $N_x$ such that $\Omega \cap N_x$ can be written as :
\begin{equation}
\Omega \cap N_x = V^{2} \cup V^1 \cup V^0,
\end{equation}
where $V^2$ is a finite union of analytic $2$-submanifolds, $V^1$  a finite union of analytic $1$-submanifolds and $V^0$ is a finite union of points.
By compactness of $\Omega$, we can write it as: 
\begin{equation}\label{strateq}
\Omega= V^{2} \cup V^1 \cup V^0,
\end{equation}
where $V^2$ is a finite union of analytic $2$-submanifolds, $V^1$  a finite union of analytic $1$-submanifolds and $V^0$ is a finite union of points.

Noting that $V^1 \cup V^0 \subset \overline{V^2}$, we have:

\begin{equation}\label{supportspheres}
\Omega \subset \bigcup_{i} S_i
\end{equation}

We now proceed to prove that $M_i = S_i$ for all $i$ and $\Omega= \mathcal{R}$.

Suppose that there exists $i$ such that $M_i \neq S_i$, and assume without loss of generality that $i=1$.
Pick $y \in \partial(\Omega \cap S_1)$ (by $ \partial(\Omega \cap S_1)$ we mean the boundary in the subspace topology of $S_1$ in the following).
We first claim that $y \in \cup_{i \neq 1} S_i$. Suppose not. Then there exists $\epsilon$ such that $B(y,\epsilon) \subset \left( \bigcup_{i \neq 1} S_i \right)^{c}$.
In particular, by $\eqref{supportspheres}$, 
$$ B(y,\epsilon) \cap \Omega = B(y, \epsilon) \cap \Omega \cap S_1.$$
On the other hand, since $y \in  \partial(\Omega \cap S_1)$, $ B(y,\epsilon) \cap \Omega^c \cap S_1$ is a non-empty open subset of $S_1$ and consequently, $\mathcal{H}^{2} ( B(y,\epsilon) \cap \Omega^c \cap S_1) >0$.
Thus we have, since $\mathcal{H}^{2} \res \Omega$ and $\mathcal{H}^{2} \res S_1 $  are support $2$-uniform, 
\begin{align*}
\pi \epsilon^{2} = \mathcal{H}^{2} ( B(y,\epsilon) \cap \Omega) &= \mathcal{H}^{2}(B(y, \epsilon) \cap \Omega \cap S_1), \\ & < \mathcal{H}^{2}(B(y, \epsilon) \cap \Omega \cap S_1) + \mathcal{H}^{2}(B(y, \epsilon) \cap \Omega^{c} \cap S_1), \\ &= \mathcal{H}^{2}(B(y, \epsilon)  \cap S_1) = \pi \epsilon^{2},
\end{align*}
which yields a contradiction.
Hence, $y \in \bigcup_{i \neq 1} S_i$.
In other words, for each $y \in \partial(\Omega \cap S_1)$ there exists a finite index set $I$, $1 \in I$, such that $y \in \Omega \cap \bigcap_{i \in I} S_i$.
We now prove that such a set consists of a unique point. Let $e \in \Omega \cap \bigcap_{i \in I} S_i$. The fact that $\overline{V^1 \cup V^{0}} \subset V^{2} $ in $\eqref{strateq}$ means that, for every $i$, $S_i \cap \Omega =  M_i \cup \partial M_i$.
 In particular, since $e \in S_i \cap \Omega$, there exists a sequence of points $e_l \in M_i$ for some $M_i$ with $e \in M_i$ ( possibly all identified with $e$) such that $e_l \to e$. But $\Omega$ being a $C^{1,\alpha}$ submanifold (see Theorem [1.3] from $\cite{N}$), the tangent planes $T_{e_l} \Omega$ converge to $T_{e} \Omega$. On the other hand, $T_{e_l} S_i$ also converge to $T_{e} S_i$. Since $T_{e_l} S_i = T_{e_l} \Omega$, we get $T_{e} \Omega= T_{e} S_i$, for all $i$. In other words the spheres $S_i$, for $i \in I$, are tangent in $e$.
This implies that $\partial(\Omega \cap S_1)$ is a finite union of points. Therefore, $\Omega \cap S_1$ is a finite union of points. So any sphere $S_i$ such that $M_i \neq S_i$ only intersects $\Omega$ in a finite union of points. Since it is clear that two spheres cannot intersect in points from $M_i$  ($\Omega$ being the support of a support $2$-uniform measure), we can exclude a sphere intersecting $\Omega$ in a discrete set from our decomposition $\eqref{supportspheres}$. This ends the proof that for every $i$, $M_i=S_i$.

In particular, the spheres are disjoint and  $\Omega = \mathcal{R}$, since $\partial M_i = \emptyset$, for all $i$.

\end{proof}

We end this subsection by proving two simple lemmas about $\Omega$ which will be useful in describing the $2$-spheres composing it.

.

\begin{lemma} \label{isolatedspheres}
For $i>0$, let $r_i$ be the radius of $S_i$. Then if $e \in S_i$, we have:
\begin{equation} \label{layer0}
B(e,2r_i) \cap (\Omega \backslash S_i) = \emptyset
\end{equation}
\end{lemma}
\begin{proof}
Let $\rho \in (0,2r_i)$.
Clearly, $r_i <1$  since $S_i$ is a subset of $\mathbb{S}^{d-1}$. By Theorem $\ref{3unif}$,
\begin{equation}
\sigma(B(e,\rho))=\pi \rho^{2}.
\end{equation}
On the other hand,
\begin{align*}
\sigma(B(e,\rho)) &= \mathcal{H}^{2}(B(e,\rho) \cap \Omega), \\ &= \mathcal{H}^{2}(B(e,\rho) \cap S_i) + \mathcal{H}^{2} (B(e,\rho) \cap (\Omega \backslash S_i)), \\
&=\pi \rho^{2} + \mathcal{H}^{2} (B(e,\rho) \cap (\Omega \backslash S_i)) \mbox{ , by } \eqref{spherelocally2unifeq}.
\end{align*}
In particular, 
\begin{equation}
\mathcal{H}^{2} (B(e,\rho) \cap (\Omega \backslash S_i))=0
\end{equation}
Assume there exists $x \in B(e,\rho) \cap \Omega \backslash S_i $. Then there exists $\delta>0$ such that 
\begin{equation}
\Omega \cap B(x,\delta) \subset B(e,\rho) \backslash S_i
\end{equation}
and consequently
$$\mathcal{H}^{2} (B(e,\rho) \cap (\Omega \backslash S_i)) > \mathcal{H}^{2} (\Omega \cap B(x,\delta))>0,$$
yielding a contradiction.

\end{proof}

\begin{lemma}\label{distancespheres}
For $i >0$, $e \in S_i$, there exists $z \in \Omega \backslash S_i$ (not necessarily unique) such that:
$$|z-e|= 2r_i.$$

In particular, this combined with $\eqref{layer0}$ implies that $dist(e, \Omega \backslash S_i)=2r_i$.
\end{lemma}
\begin{proof}
For $\epsilon>0$ small enough, 
\begin{equation}
\sigma(B(e, 2r_i (1+\epsilon)))- \sigma(S_i)= 4\pi r_i^2 \epsilon (2+\epsilon)>0.
\end{equation}
On the other hand,
\begin{equation}
\sigma(B(e, 2r_i (1+\epsilon)))- \sigma(S_i)= \mathcal{H}^{2}(\left(\Omega \cap B(e, 2r_i (1+\epsilon)) \right)\backslash S_i).
\end{equation}
In particular, for all $j>0$, $j$ large enough, $$\left(\Omega \cap B(e, 2r_i (1+\frac{1}{j})) \right)\backslash S_i  \neq \emptyset,$$
and there exists $z_j \in \left(\Omega \cap B(e, 2r_i (1+\frac{1}{j})) \right)\backslash S_i$.
Passing to a subsequence if necessary, $z_j \to z$, $z \in \Omega$, $|z-e|= 2r_i$.
Moreover, $z \notin S_i$. 
If it were, then for $j$ large enough, $dist(z,z_j) < 2 r_i$ contradicting $\ref{layer0}$. 

\end{proof}

\subsection{Configuration of the 2-spheres and distance symmetry}

We now want to obtain a better description of the spheres that compose the support of $\Omega$. We start with two lemmas of elementary geometry.

\begin{lemma} \label{geometriclemma}
Let $S$ be a two $2$-dimensional sphere in $\RR^d$ such that $S \subset T$, where $T$ is an affine $3$-plane.
We let $e \in \RR^{d}$ and follow the notation $d(e,S)=D$, $r(S)=\rho$ and $d(e,T)=\delta$.
Then, for $D<R$:
\begin{equation}
B(e,R) \cap S = B(p,x) \cap S,
\end{equation} 
where $\left\lbrace p \right\rbrace = B(e,D) \cap S$, 
$$ x^{2} = \frac{\rho}{\rho + (-1)^{sgn(e)} (D^2 - \delta^2)^{\frac{1}{2}}} (R^2 - D^2)$$
where $sgn(e)$ is $0$ if the orthogonal projection of $e$ on $T$ lies outside $S$, and $1$ otherwise.
\end{lemma}
\begin{proof}
Let $f$ be such that: 
$$ B(e, \delta) \cap T = \left\lbrace f \right\rbrace.$$
We assume for simplicity that $f$ lies outside of $S$.
Then:
$$B(e,R) \cap T= B_3 (f, \tilde{R}),$$
and 
$$B(e,D) \cap T = B_3(f, \tilde{D}), $$
where $B_3$ denotes the three-dimensional ball in $T$,  $R^2 = \tilde{R}^2 + \delta^{2}$ and $D^2 = \tilde{D}^2 + \delta^{2}$.
Also note that $B_{d}(e,R) \cap T \cap S = B_{3} (f,\tilde{R}) \cap S$ since $S \subset T$.

Let $q$ be the center of $S$.
Then $f$, $p$ and $q$ are aligned since $S$ and $\partial B_{3}(E, \tilde{D})$ are tangent at $p$.

Moreover, $B(f,\tilde{R})$ and $S$ intersect in a circle $C$. 
For any $u,v \in C$, $|p-u|=|p-v|=x$. Indeed, since $|f-u|=|f-v|= \tilde{R}$, $|q-u|=|q-v|= \rho$, and $f$, $p$, $q$ aligned, $p$ is in the bisecting plane of any two such points.
Therefore,
$$B_{d}(p,x) \cap S = B_{3}(p,x) \cap S = B_{3}(f, \tilde{R}) \cap S = B_{d}(e,R) \cap S. $$
To end the proof, we compute $x$.
Choose $m \in C$ and let $n$ be its projection on the line $(fq)$.
We work in the $2$-plane $T_2$ containing $f$, $q$ and $m$.
Then $|p-m|=x$, $|q-m|=|p-q|=\rho$, $|f-p|=\tilde{D}$ and $|f-m|= \tilde{R}$.
Moreover, we denote $|m-n|$ and $|p-n|$ by $l$ and $t$ respectively.
Then, applying Pythagoras' theorem, we get:
\begin{align}
\rho^{2}&= l^2 + (\rho - t)^2, \label{tri1}\\
x^2 &=l^2 + t^2,\label{tri2} \\
\tilde{R}^2 & = l^2 + (\tilde{D}+t)^2 . \label{tri3}
\end{align}
Then $\eqref{tri1}$ becomes $l^2= 2\rho t - t^2$ and plugging this into $\eqref{tri2}$ gives \begin{equation} \label{tri2'}
x^2= 2 \rho t,
\end{equation}
 and $\eqref{tri3}$ becomes $$t= \frac{\tilde{R}^2 - \tilde{D}^2}{2(\rho + \tilde{D})}.$$
Finally, $\eqref{tri2'}$ gives:
\begin{equation}
x^2= \frac{\rho}{\rho+ \tilde{D}} (\tilde{R}^2 - \tilde{D}^2).
\end{equation}
Expressing $\tilde{R}$ and $\tilde{D}$ in terms of $R$, $D$ and $\delta$ ends the proof.
Note that if $f$ lied inside $S$, the same reasoning would have given $x^2= \frac{\rho}{\rho- \tilde{D}} (\tilde{R}^2 - \tilde{D}^2)$.
\end{proof}

\begin{lemma}\label{closestpoint}
Let $S$ be the $2$-sphere in $\mathbb{S}^{d-1}$ defined by:
\begin{equation}\nonumber
S= \left\lbrace z \in \mathbb{S}^{d-1}; |z-\xi|=r, z \in V+\xi \right\rbrace,
\end{equation}
where $V$ is a linear $3$-plane. 
Then for all $z \in \mathbb{R}^d$, if $P_V(z) \neq 0$, denoting the closest point to $z$ and furthest point to $z$ on $S$ by $P_S$ and $\overline{P}_S$, we have:
\begin{equation}
P_S(z)=r \frac{P_V(z)}{|P_V(z)|}+ \xi,
\end{equation}
and 
\begin{equation}
\overline{P}_S(z)=-r \frac{P_V(z)}{|P_V(z)|}+ \xi,
\end{equation}
where $P_V$ is the linear projection on $V$.
We also denote by $D_{S}(z)$ the distance from $z$ to $S$ and $\overline{D_{S}}(z)$ the distance between $z$ and the furthest point to $z$ on $S$.
\end{lemma}
\begin{proof}
We start by proving that $P_{V+\xi}(z)=P_{V}(z)+\xi$, where $P_{V+\xi}$ denotes the affine projection on $V+\xi$.
First note that $\xi$ is normal to $V$.
Indeed, if $e$ is a unit vector of $V$, we have $|\xi+re|=|\xi-re|=1$ since $\xi+re$ and $\xi-re$ are points of $S \subset \mathbb{S}^{d-1}$. This gives
$$\left\langle \xi , \xi+re \right\rangle = \left\langle \xi , \xi - re \right\rangle,$$
and consequently $\xi . e =0$.

$P_{V+\xi}(z)$ is the point $\tilde{e} $  that minimize $|z-\tilde{e}|$ for $\tilde{e} \in V+\xi$. Writing $\tilde{e}=e+\xi$, $P_{V+\xi}(z)=e+\xi$ where $e$ minimizes $|z-e-\xi|, e \in V$.
But, since $$ |z-e-\xi|^2 = |P_V(z)-e|^2 + |P_{V^{\perp}}(z) - \xi |^{2},$$ it is clear that $e=P_V(z)$ is the minimizer we're looking for. This proves that $P_{V+\xi}(z)=P_{V}(z)+\xi$.

Now if $u \in S \subset V+\xi$ minimizes (resp. maximizes) $|z-u|$, by writing $$|z-u|^2 = |P_{V+\xi}(z)-u|^{2} + |P_{V+\xi}^{\perp}(z)|^{2},$$
we see that $u$ minimizes (resp. maximizes) 
$|P_{V+\xi}(z)-u|=|P_{V}(z)-(u-\xi)|$ and consequently $u$ maximizes (resp. minimizes) $\left\langle P_{V}(z) , u-\xi \right\rangle.$
Therefore, $\frac{u-\xi}{r}=\frac{P_V(z)}{|P_V(z)|}$
(respectively, $\frac{u-\xi}{r}=- \frac{P_V(z)}{|P_V(z)|}$).
\end{proof}

Using Lemma $\ref{geometriclemma}$, Lemma $\ref{closestpoint}$ and the fact that $\sigma$ is  support $2$-uniform, we deduce the following technical lemma which will be our first step towards a description of the spherical component.

\begin{lemma}\label{lemmacharactlocally2unif}
Let $\Omega \subset \mathbb{S}^{d-1}$, and $\sigma = \mathcal{H}^{2} \res \Omega$. Assume that $\sigma$ satisfies:
\begin{equation}\label{locally2uniformcond}
\sigma(B(x,r))=\pi r^2,
\end{equation}
for every $0 \leq r \leq 2$, for every $x \in \Omega$.
From Theorems $\ref{finitelymanyspheres} $ and $\ref{algvarcone}$ we know that $\Omega= \bigcup_{i=1}^{M} S_i$ where $S_i$ is a $2$-sphere of radius $r_{S_i}$. Let $\mathfrak{S}= \bigcup_{i=1}^{M} \left\lbrace S_i \right\rbrace$ and fix  $z \in \Omega$. Define the integer $m(z)$, the indices $\left\lbrace i \right\rbrace_{i=1}^{m(z)}$, the radii $\left\lbrace  R_i(z) \right\rbrace_{i=1}^{m(z)}$ and the subsets $ \left\lbrace C^{i}(z) \right\rbrace_{i=1}^{m(z)}$, $\left\lbrace C_{i}^{j}(z) \right\rbrace_{0 \leq j \leq i \leq m(z)}$ of $\mathfrak{S}$ inductively in the following manner
\begin{itemize}
\item $R_{1}(z)=2r_{z}$ where $r_z$ is the radius of the sphere $S_z$ such that $z \in S_z$.
\item $C^{0}(z)=C_{0}^{0}(z)=\left\lbrace S_z \right\rbrace$,
\item The first layer $C^{1}(z)=C^{1}_{1}(z)= \bigcup \left\lbrace \left\lbrace S \right\rbrace ; D_{S}(z)= R_1(z)\right\rbrace $ and the contribution of the zero-th to the first layer $C^{1}_{0}(z)=\emptyset$,
\item If $1 \leq i $, $R_i(z)= \inf \left\lbrace \overline{D}_{S}(z); S \in C^{i-1}(z) \right\rbrace$, and $C^{i}_{i}(z)=\bigcup \left\lbrace \left\lbrace S \right\rbrace ; D_{S}(z)= R_{i}(z) \right\rbrace$.
\item For $0 \leq j \leq i$, the contribution of the $j$-th layer to the $i$-th layer $$C_{j}^{i}(z)= \bigcup_{S \in C^{j}(z)} \left\lbrace \left\lbrace S \right\rbrace ; \overline{D}_{S}(z) > R_{i}(z) \right\rbrace.$$
\item $C^{i}(z)= \bigcup_{0 \leq j \leq i} C_{j}^{i}(z)$.
 \item $m(z)$ to be the first integer so that $R_{m(z)} =2$ and $C_{j}^{m(z)} = \emptyset$ for all $j \leq m(z)$.
\end{itemize}
Then, $\Omega = - \Omega$ and for every $z$, letting $$c_{S}(z)= \frac{r_S}{r_S+(-1)^{sgn(z)} \left({D_S(z)}^{2} - {\delta_{S}(z)}^{2} \right)^{\frac{1}{2}}},$$ 
where $\delta_{S}(z)$ is the distance from $z$ to the affine $3$-plane containing $S$, we have for every $0 \leq i \leq m(z)$, \begin{equation} \label{cSDS}
4 \sum_{1 \leq j \leq i} \sum_{S \in C^{j-1} \backslash C^{j}} r_{S}^2 = \sum_{S \in C^{i}(z)} c_{S}(z) D_S(z)^2
\end{equation}
and 
\begin{equation}\label{cS}
\sum_{S \in C^{i}(z)} c_{S}(z) =1.
\end{equation}
In particular, for every $0<i< m(z)$, $C^{i}(z) \neq \emptyset$ and $\Omega= \bigcup_{0 \leq i \leq m(z)} \bigcup_{S \in C^{i}} S$.

\end{lemma}

\begin{proof}
By Lemma $\ref{finitelymanyspheres}$, we know that $\Omega= \bigcup_{i=1}^{M} S_i$ and $\Omega = - \Omega$.
Fix $z \in \Omega$. By Lemmas $\ref{distancespheres}$ and $\ref{layer0}$, we know that $C^{1}(z) \neq \emptyset$.
For any $i$, if $S \in C^{i}(z)$, then $D_{S}(z) \leq R_{i}(z)$ and $\overline{D}_{S}(z) > R_{i}(z)$ so that whenever $S \in C^{i}(z)$ and $R_i(z) < R < R_{i+1}(z)$, we have $S \cap B(z,R) \neq \emptyset$ and $S \cap \left( B(z,R) \right)^{c} \neq \emptyset$.
Moreover, if $S \in \bigcup_{l \leq i}  C^{l-1} \backslash C^{l}$, then $\overline{D}_S(z) \leq R_{i}(z)$.
Hence, for $R_i(z) < R < R_{i+1}(z)$, 
\begin{equation}
B(z,R) \cap \Omega = \left( \bigcup_{l=1}^{i} \bigcup_{S \in C^{l-1}(z) \backslash C^{l}(z)} S \right) \bigcup \left( \bigcup_{S \in C^{i}(z)} S \cap B(z,R) \right)
\end{equation}
Applying $\mathcal{H}^{2}$ on both sides, we get from the fact that $\sigma$ is support $2$-uniform and by Lemma $\ref{geometriclemma}$ and $\eqref{spherelocally2unifeq}$,
\begin{equation} \label{measurelayer}
\pi R^{2} = \sum_{l=1}^{i} \sum_{S \in C^{l-1}(z) \backslash C^{l}(z)} 4\pi r_{S}^2  + \sum_{S \in C^{i}(z)} \pi c_S(z) (R^2 - D_S(z)^2).
\end{equation}
Differentiating twice with respect to $R$ gives $\eqref{cS}$ and plugging $\eqref{cS}$ back into $\eqref{measurelayer}$ gives $\eqref{cSDS}$.
Note that $\eqref{cS}$ directly implies that every $C^{i}$ is non-empty since $c_{S}(z)>0$ for every $S$.

\end{proof}
We now use this theorem to prove that the support of a support $2$-uniform measure is symmetric in a sense that will be made precise. Let us start by defining a notion of symmetry for points.

\begin{definition} \label{layers}
Let $\mathcal{L}=\left\lbrace l_i \right\rbrace_{i=0}^{m-1} \subset S_{m}$ be a set of permutations that satisfies the following:
 for each $i$, $l_i$ has the following properties
\begin{enumerate}
\item $l_0(j)=j$,
\item $l_i(1)=i+1$,
\item For all $i \neq k$, for all $j$, $l_{i}(j) \neq l_{k}(j)$.
\item  $l_i^{-1}= l_i$.
\end{enumerate}
We call such an $\mathcal{L}$ a layering and the permutations in that set are called layering functions or permutations.

If $r>0$ and $\left\lbrace \alpha_{1}, \ldots, \alpha_{m} \right\rbrace$ is a set of points in $\mathbb{R}^{d}$ such that: 
\begin{equation}\label{welllayered6}
| \alpha_{j} - \alpha_{l_i(j)}|=2 \sqrt{i}r, \mbox{ for all } j, i
\end{equation}
then we call it an $r$-distance symmetric set of points. In the case where $r=\frac{1}{\sqrt{m}}$, we say the set is distance symmetric.

Finally, for $1\leq i < j \leq m$, we denote by $d_{ij}$ the integer such that: \begin{equation}
l_{d_{ij}}(i)=j,
\end{equation}
and set $$d_{ii}=0,$$ for all $i$. We call the function $d$ such that $d(i,j)=d_{ij}$ the distance function of $\mathcal{L}$.
\end{definition}

\begin{remark}
\begin{enumerate}
\item If $\left\lbrace \alpha_{j} \right\rbrace $ is  $r$-distance symmetric and the associated permutations are $\left\lbrace l_i \right\rbrace$ then for all $j$, $$\left\lbrace j, l_1(j), \ldots, l_{m-1}(j) \right\rbrace = \left\lbrace 1, \ldots, m \right\rbrace.$$

\item The $l_i$'s organize the points of $P$ into layers. Let $P_j$ be the sequence:
$$P_j= (\alpha_{j}, \alpha_{l_1(j)}, \ldots, \alpha_{l_{m-1}(j)} ).$$
Each $P_j$ is a rearrangement of $P_1$``viewed through the lens" of $\alpha_j$:  $\alpha_{l_i(j)}$ is the $i$-th layer of $P_j$ and is at a distance $2\sqrt{i}r$ from $\alpha_j$.
\end{enumerate}
\end{remark}

\begin{theorem} \label{suffcond}
Let $\Omega \subset \mathbb{S}^{d-1}$, $\sigma= \mathcal{H}^{2} \res \Omega$. Assume $\Omega$ is a union of distance symmetric $2$-spheres i.e. $\Omega=\left( \bigcup_{i=1}^{2m} S_i \right)$ where: 
\begin{enumerate}
\item For $i=1, \ldots, 2m$, $S_i$ is the $2$- sphere of radius $r=\frac{1}{\sqrt{2m}}$ and center $\xi_{i}$,
\item For all $i=1,\ldots,2m$, $S_i \subset V+ \xi_i$ where $V$ is a linear $3$-plane such that $\xi_{i} \in V^{\perp}$.
\item $\left\lbrace \xi_i \right\rbrace_{i=1}^{2m}$ is a distance symmetric set of points in $V^{\perp}$.
\end{enumerate}
Then 
\begin{equation}\label{locally2uniformeq}
\sigma ( B(x,r) ) = \pi r^{2}, \mbox{ for } x \in \Omega \; , \; 0 \leq r \leq 2.
\end{equation}

\end{theorem}

\begin{proof}

We first claim that if $\Omega$ is a as described in the statement of the theorem, then for fixed $j$, for all $z \in S_j$, for all $i$ we have:
\begin{equation}\label{distanceforlayeredunion} D_{S_{l_{i}(j)}}(z) = 2 \sqrt{i} r = \overline{D_{S_{l_{i-1}(j)}}}(z).
 \end{equation}

We prove it for $j=1$. The proof for other $j$'s is exactly similar.
First note that by hypothesis, we have $P_{V^{\perp}}(z)=\xi_{1}$ for $z \in S_1$.
Moreover $|P_{V}(z)|=| P_{V+\xi_{1}}(z) - \xi_{1}|=r$.
Thus
\begin{align*}
{D_{S_{i}}(z)}^{2} &= |z - P_{S_i}(z)|^{2}, \\ 
&= | P_{V}(z) + \xi_{1} - P_{V}(z) - \xi_{i} |^{2} , \\ &= |\xi_{1} - \xi_{i}|^{2}, \\&=4 (i-1)r^{2},
\end{align*}
and
\begin{align*}
\overline{D_{S_{i-1}}(z)}^{2} & = |z - \overline{P_{S_i}}(z)|^{2}, \\ 
&= | P_{V}(z) + \xi_{1} + P_{V}(z) - \xi_{i-1} |^{2} , \\&= 4 |P_V(z)|^{2}+ |\xi_{1} - \xi_{i-1}|^{2}, \\&=4r^{2}+4 (i-2 ) r^{2}, \\ &= 4(i-1)r^{2}.
\end{align*}
We now show that $\eqref{locally2uniformeq}$ holds. Pick $z \in \Omega$. Without loss of generality, we can assume that $z \in S_1$.
Let $0 \leq R \leq 2$.
Then there exists $i$ such that $2 \sqrt{i}r \leq R \leq 2 \sqrt{i+1} r $.
If $R =2 \sqrt{i}r$, then by Lemma $\ref{lemmacharactlocally2unif} $,
$B(z,R) \cap \Omega = \bigcup_{k=1}^{i} S_{k}$ and \begin{align*}
\mathcal{H}^{2}(B(z,R) \cap \Omega) &= \sum_{k=1}^{i} \mathcal{H}^{2}(S_k), \\&= \sum_{k=1}^{i} \pi 4 {r_{S_k}}^{2} \mbox{ , by} \eqref{spherelocally2unifeq} \\&= \pi \left(4 \sum_{k=1}^{i} r^{2} \right), \\ &=4 \pi ir^{2} \\&= \pi R^{2}.
\end{align*}
If $2 \sqrt{i}r< R < 2 \sqrt{i+1}r$,
then $B(z,R) \cap \Omega = \left( \bigcup_{k=1}^{i} S_{k} \right) \bigcup \left( S_{i+1} \cap B(z,R) \right)$ and \begin{align*}
\mathcal{H}^{2}(B(z,R) \cap \Omega)&= \sum_{k=1}^{i}  \mathcal{H}^{2}(S_k) + \mathcal{H}^{2}(S_{i+1}\cap B(z,R)), \\&=\pi {D_{S_{i+1}}(z)}^{2} + \pi (R^{2} - {D_{S_{i+1}}(z)}^{2}), \mbox{ by Lemma } \ref{geometriclemma} \mbox{ and } \eqref{spherelocally2unifeq}, \\ &= \pi R^{2}. 
\end{align*}

\end{proof}
\subsection{Classification in codimension 2}
We prove that in codimension $2$, all conical $3$-uniform measures come from a set of distance symmetric spheres.
\begin{theorem}\label{necessarycond}
Let $\sigma$ be the spherical component of a conical $3$-uniform measure in $\RR^{5}$, $\Omega = \supp(\sigma)$. Then $\Omega$ is a union of distance symmetric $2$-spheres.
\end{theorem}

\begin{proof}
Suppose $\Omega \neq \mathbb{S}^{2} \times \left\lbrace 0 \right\rbrace$ , assume $S_1$ is the sphere with smallest radius and denote $r_{S_1}$ by $r_1$. 

Then $r_1 \leq \frac{\sqrt{2}}{2}$: indeed, on one hand the sum of the squares of the radii is $1$ since for any $z \in \Omega$, 
$$4\pi =\sigma(B(z,2))= 4\pi \sum_{S \in \mathfrak{S} } {r_{S}}^{2},$$ and on the other hand the fact that $\Omega= -\Omega$ and $\Omega \neq \mathbb{S}^{2} \times \left\lbrace 0 \right\rbrace$  implies that there are at least $2$ $2$-spheres in $\Omega$.

If $r_1=\frac{\sqrt{2}}{2}$, then $D_{S_2}=2r_1=\sqrt{2}$ which implies that $S_2=-S_1$ and $\overline{D_{S_2}}=2$. Therefore, $\Omega=S_1 \cup ( - S_1)$ which ends the proof. 

From now on, we assume that  $r_1 < \frac{\sqrt{2}}{2}$.

Assume that $S_1 \subset \left(V_1+\xi \right) \cap \mathbb{S}^{d-1} $ where $V_1$ is a linear  $3$-plane normal to $\xi$ and $\xi$ is the center of $S_1$.


If a point $z \in S_1$ in $\Omega$ is chosen, the layered character of the support allows us, for every other $2$-sphere
$S$ in $\Omega$, to write an equation for $z$ in terms of the $V_S$'s, $\xi_{S}$'s and $r_S$'s, the plane, center and radius of $S$. The fact that these equations are quadric and that $z$ is already assumed to be in the quadric $S_1$ will allow us to relate $S$ to $S_1$. 

Our first step will be to write these equations if $S$ is assumed to be in the first layer of $z$. 
To this end, set $C^{1}(1)= \bigcup_{z \in S_1} C^{1}(z)$ the first layer with respect to $S_1$, pick $S \in C^{1}$ with radius $r$ and center $\eta$. We can write $S \subset V+\eta $, for some linear $3$-plane $V$ normal to $\eta$. Set $A_S = \left\lbrace z \in  V_{1}+\xi \; ; \; |z - P_{S}(z)|=2r_1 \right\rbrace$.
We wish to write an equation for $A_S$ as an object in the $3$-space $V_{1}+\xi$. Choose orthonormal bases $ \left\lbrace e_i \right\rbrace_{i=1}^{3}$ of $V_1$ and $ \left\lbrace u, v, w \right\rbrace$ for $V$ and write $\xi=te$. 
In the following we will denote $\left\langle z,e_i \right\rangle$ by $z_i$ and $ \left\langle z,e \right\rangle$ by $z_e$.

On one hand, we have

\begin{align} \label{equationquadric1}
|z - P_{S}(z)|^{2} &= 4 r_1^2, \nonumber \\
|P_{V}(z)|^{2}(1 - \frac{r}{|P_{V}(z)|})^{2}+ |P_{V^{\perp}}(z)-\eta |^{2} &= 4r_1^2, \mbox{ by Lemma } \ref{closestpoint}, \nonumber \\
|P_{V}(z)|^{2}+ r^{2}- 2r |P_{V}(z)| + |P_{V^{\perp}}(z)|^{2} + |\eta|^{2} - 2 \left\langle \eta, z \right\rangle &=
 4r_1^2,\nonumber  \\
 2 - 2( r |P_{V}(z)| + \left\langle \eta, z \right\rangle) &=4r_1^{2} , \mbox{ since } |\eta|^{2} + r^{2}=1, \mbox{ and } |z|=1, \nonumber\\
r |P_{V}(z)| &= K- \left\langle \eta, z \right\rangle, \mbox{ where } K= \frac{2-4r_1^{2}}{2}.
\end{align}

We first square and expand the right hand side of $\eqref{equationquadric1}$.
Writing $ \left\langle z, \eta \right\rangle =  \eta_1 z_1 + \eta_2 z_2+ \eta_3 z_3 + t \eta_{e}$ and expanding the right hand side we get:

\begin{align}\label{quadric RHS}
(K- \left\langle \eta, z \right\rangle)^{2} &= K^{2}+ t^{2} \eta_{e}^{2}- 2Kt \eta_{e} + \sum_{i=1}^{3} \eta_i^{2} z_{i}^{2} + \sum_{1 \leq j < k \leq 3} 2 \eta_{j} \eta_{k} z_{j} z_{k} + \sum_{i=1}^{3} (2t \eta_{e} - 2K) \eta_{i} z_{i},\\ &= (K +t \eta_{e})^{2} +\sum_{i=1}^{3} \eta_i^{2} z_{i}^{2} + \sum_{1 \leq j < k \leq 3} 2 \eta_{j} \eta_{k} z_{j} z_{k} + \sum_{i=1}^{3} (2t \eta_{e} - 2K) \eta_{i} z_{i}.
\end{align}

On the other hand, the left hand side becomes:

\begin{align}
r^{2} |P_{V}(z)|^{2}  &= \sum_{i=1}^{3} r^{2} (u_{i}^2 + v_{i}^{2} + w_{i}^{2})z_i^{2} + \sum_{1 \leq j < k \leq 3} 2r^{2} (u_j u_k + v_j v_k + w_j w_k)z_j z_k  \\& + \sum_{i=1}^{3}2tr^{2} ( u_e u_i + v_e v_i + w_e w_i)z_i + r^{2} t^{2} (u_{e}^{2} + v_{e}^{2} + w_{e}^{2}),\\
&= \sum_{i=1}^{3} r^{2} |P_{V}(e_i)|^{2}z_i^2 + \sum_{1 \leq  j < k \leq 3} 2 r^{2} \left\langle P_{V}(e_i), P_{V}(e_j) \right\rangle z_j z_k \\ &+ \sum_{i=1}^{3} 2 t r^{2} \left\langle P_{V}(e_i), P_{V}(e) \right\rangle z_i +r^{2} t^{2} |P_{V}(e)|^{2}.
\end{align}
$\eqref{equationquadric1}$ becomes: \begin{equation} \label{quadricquadric}
\sum_{i=1}^{3} a_i z_i^{2} + \sum_{1 \leq j < k \leq 3} b_{jk} z_j z_k + \sum_{i=1}^{3} c_i z_i = (K-t \eta_{e})^{2} - r^{2} t^{2} |P_{V}(e)|^{2},
\end{equation}
where
$$a_i = r^{2} |P_{V}(e_i)|^{2} - \eta_{i}^{2},$$
$$b_{jk}= 2r^{2}  \left\langle P_{V}(e_j) , P_{V}(e_k) \right\rangle - 2 \eta_j \eta_k,$$
and $$c_{i}= (2tr^{2} \left\langle P_{V}(e), P_{V}(e_i) \right\rangle - 2t \eta_{e} + 2K) \eta_{i}.$$
This gives us the equation of $A_S$ for $S$ in the first layer of $S_1$.

Note that $$S_1 = \bigcup_{S \in C^{1}} \left( A_S \cap S_{1} \right) .$$
Indeed, if $z \in S_{1}$, then there exists $S \in C^{1}$ so that $S \in C^{1}(z)$. In other words, $dist(z,S)= 2r_1$ or $z \in A_{S}$.
But there are only finitely many $S$'s in $C^{1}$ since $\Omega$ is a finite union of spheres. So there exists at least one $A_S \cap S_1$ of dimension 2. But two distinct quadrics can intersect only in a curve or a point if at all. This implies that $A_S$ and $S_1$ are trivially identified.
In other words we can identify the coefficients of the quadric $\eqref{quadricquadric}$ with the coefficients of the quadric
\begin{equation}\label{lambdaquadric} \sum_{i=1}^{3}  z_i^{2} =  r_1^{2},
\end{equation}
up to a multiple $\lambda \in \RR$.

We now claim that this implies that $V=V_1$. This is where the hypothesis that $d=5$ will be essential. Indeed, this will allow us to identify the multiple $\lambda$ in the identification between the quadrics.

Since $\dim(V+V_{1}) \leq  5$ , $\dim (V \cap V_{1}) \geq 1$. We can assume without loss of generality that $e_1=u$ and $e_1 \in V \cap V_{1}$.
In this case, we get $a_1 = r^{2}$ since $\eta_{1}=0$, $\eta$ being normal to $V$.
This implies that $\lambda=1$ in $\eqref{lambdaquadric}$ and consequently, 
$$ r^{2} |P_{V}(e_i)|^{2} - \eta_{i}^{2} =r^{2},$$
for $i=2,3$.

But $P_V$ being the projection on $V$, 
$$ r^{2} |P_{V}(e_i)|^{2} - \eta_{i}^{2} < r^{2},$$
unless $e_i \in V$ and $\eta_i=0$. This proves that $V=V_1$.

Most of the work is done now that $V$ and $V_1$ are known to be identical. Some additional calculations will prove that $r_S=r$ for every $S$ and that there is a unique sphere in the first layer: write $$C^1= C^1_1 \cup C^1_2$$ where 
$$ C^1_1= \left\lbrace S ; dim(A_S \cap S_1) < 2 \right\rbrace, $$
and $$C^1_2 =  \left\lbrace S ; A_S \cap S_1 = S_1 \right\rbrace.$$

 Since $V:=V_1=V_S$ for every $S \in C^{1}_2$,  we have  for $z \in S_1$ and $S \in C^{1}_2$,

\begin{align}
D_{S}^{2}(z) - \delta_{S}^{2}(z) &= |z- P_{S}(z)|^{2} - | z - P_{{V+\eta}}(z)|^{2}, \\&=|P_{V}(z) - \frac{r_{S}}{r_1}P_{V}(z)|^{2} + | P_{V^{\perp}}(z)-\eta|^2 -| P_{V^{\perp}}(z)-\eta|^2, \\ &=(r_{S} - r_{1})^2.
\end{align}

Thus 
$\left( D_{S}^{2}(z) - \delta_{S}^{2}(z) \right) ^ {\frac{1}{2}}= r_{S} - r_{1}$ since $r_S \geq r_1$ (in fact, it is easy to prove that in general codimension, for the ``first layer'' $C^1$, $r_S = r_1$, but we assume less so that the proof follows through for other layers) and


\begin{equation}
c_{S} =\frac{ r_{S}} {2r_{S}  - r}
 \end{equation}
Note that for $S \in C^{1}_2$, $c_S(z)$ is independent of $z$ and
\begin{equation}\label{cSbig}
 c_{S} > \frac{r_{S}}{2 r_{S}}=\frac{1}{2}
\end{equation}.

By $\eqref{cSbig}$, $$\sum_{S \in C^{1}(z)} c_S > \left( \# C^{1}_{2} \right)\; . \; \frac{1}{2},$$
which implies that $C^{1}_2$ contains only one sphere. Call it $S_2$.

Now, pick $z \in S_1 \backslash \left( \cup_{S \in C^{1}_1}A_S \cap S_1 \right) $. By the above, $z$ has only $S_2$ in its first layer. Therefore, $c_{S_2}(z)=1$ which implies that $r_{S_2} = r$. 
We finally deduce that $C_1^1 = \emptyset$.
Indeed, suppose that $S \in C_1^1$ and $z \in A_S \cap S_1$. Since $S_2$ is also in the first layer of $z$, we have:

$$ 1= \sum_{S \in C^{1}(z)} c_S \geq c_{S_2}(z) + c_{S}(z) >1,$$
yielding a contradiction. 
This proves that $C^{1} =\left\lbrace  S_2 \right\rbrace$.

This proves our claim that the first layer is composed of a unique $2$-sphere $S_2$ of radius $r$ and such that $V_{S_2}=V_1$.

Note that since $C^{1}$ is composed of a unique sphere of radius $r_1$  contained in  $V_1+ \xi_{i}$ we have:

$${D_{S_2}}(z)= |P_{V}(z) + \xi_{1}-P_{V}(z)- \xi_{2}|=|\xi_1 - \xi_{2}|.$$
In particular, 
$$|\xi_1 - \xi_{2}|= 2r.$$

Moreover, for all $z$,
\begin{align*} \overline{D}_{S_2}(z)^{2}& = |z- \overline{P_{s}}(z)|^{2}, \\ &=4|P_{V}(z)|^{2} + | \xi_{1} - \xi_{2} |^{2}, \\ & = 8r^{2}. 
\end{align*}

It is now easy to repeat the proof for other layers: suppose that for some $k$, and all $i \leq k $ 
\begin{enumerate}
\item There exists a unique sphere $S_{i+1}$ such that for all $z \in S_1$, $C^{i}(z)= \left\lbrace S_{i+1} \right\rbrace$,
\item $S_{i+1} \subset V+ \xi_{i+1}$,
\item $r_{S_{i+1}} = r$,
\item For all $z \in S_1$,  $D_{S_{i+1}}(z)=\overline{D}_{S_i} (z) = | \xi_{1} - \xi_{i+1}|=2\sqrt{i}r$.
\end{enumerate}

Then repeating the exact same proof as for $S_2$, while replacing $D_{S_2}=2r$ with $D_{S_{i+1}}=2 \sqrt{i}r$, we get the same result for $S_{i+2}$.

Note that since $\Omega= -\Omega $, we have $\Omega = \left( \bigcup_{i=1}^{m} S_i \right) \bigcup \left( \bigcup_{i=1}^{m} - S_{i} \right) $.


Moreover for $i \leq m$, we have
\begin{align*}
| \xi_{1} + \xi_{i}| ^{2} &=2( |\xi_{1}|^{2} + | \xi_{i}|^{2}) -  | \xi_{1} - \xi_{i}|^{2}, \\ &=4 t^{2} - 4 ir^{2}, \\ &= 4 ( 2m-1-i) r^{2}, \\ &=4(i-1)r^{2}.
\end{align*}
 We rename $-S_{i}$ to be $S_{2m+1-i} $ for $i \leq m$.

We can now prove that $\left\lbrace \xi_{j} \right\rbrace$ is  distance symmetric: choose any $j$ and let $l_i(j)$ be such that $C^{i}(j)= \left\lbrace S_{l_i(j)} \right\rbrace$, where $C^{i}(j)$ is the $i$-th layer with respect to $S_j$. Since the spheres all have same radius the same proof as for $S_1$ can be repeated to show that for all $z \in S_i$, $D_{S_{l_j(i)}}(z) = 2 \sqrt{j} r$ and $| \xi_{i} - \xi_{l_j(i)}|=D_{S_{l_j(i)}}(z)$. 
$(1)$ and $(2)$ from Definition $\ref{layers}$ are obvious. We prove that $l_{i}$ is bijective for every $i$. Indeed this follows from the fact that
for all $j$, $$\bigcup_{i=0}^{2m-1} C^{i}(j) = \left\lbrace S_1, \ldots, S_{2m} \right\rbrace = \left\lbrace S_{j} , S_{l_1(j)}, \ldots, S_{l_{2m-1}(j)} \right\rbrace.$$
To show $(3)$ from Definition $\ref{layers}$, suppose that there existed $j$ such that $l_{i}(j) = l_{k}(j)$. Then $C^{i}(j)= C^{k}(j)$ which would in turn imply that 
$D_{S_{l_i(j)}} = 2 \sqrt{i}r = 2\sqrt{k}r=D_{S_{l_k(j)}}$ and $i=k$.
Finally $$| \xi_{j} - \xi_{{l_i}(j)}| = 2\sqrt{i}r \implies l_{i} \circ l_{i} (j) = j.$$
This proves that the centers are distance symmetric.

Finally to see that $r = \frac{1}{\sqrt{2m}}$,  fix $z \in \supp(\sigma)$ and consider $B(z,2)$. We have $4 \pi = \sigma( B(z,2)) =\sum_{i=1}^{2m} \mathcal{H}^{2} ( S_i ) = 8 m \pi r^{2}$ from which the claim follows.

\end{proof}

As a consequence we get a classification of conical $3$-uniform measures in $\RR^{5}$.
We first need to prove a lemma stating that a set of distance symmetric points is the support of a discrete uniformly distributed measure.
\begin{lemma}\label{centersunifdist}
Let $\left\lbrace \xi_{i} \right\rbrace_{i=1}^{}\subset \RR^{d}$ be a distance symmetric set of points. Then for any $c>0$, the measure 
$$\lambda= c \sum_{i=1}^{m} \delta_{\left\lbrace \xi_{i} \right\rbrace}$$
is uniformly distributed.

\end{lemma}
\begin{proof}
Fix $i$. Then, for all $0 \leq j <m-1$, for $ 2 \sqrt{j}r \leq r \leq 2 \sqrt{j+1}r$, 
\begin{align} \label{layeredunif}
\lambda(B({\xi}_i, r)) &= \lambda( \left\lbrace \xi_{l_{k}(i)} \right\rbrace_{1 \leq k \leq j}), \\
&=cj,
\end{align}
and if $r>2\sqrt{m-1}r$, $\mu(B({\xi}_i, r)=cm$.

\end{proof}

\begin{theorem} 
Let $\nu$ be a conical Radon measure in $\RR^{5}$ (i.e for all $r>0$, $\supp(\nu)=r \supp(\nu)$) and let $\Sigma:=\supp(\nu)$. Then $\nu$ is a $3$-uniform measure if and only if there exists $c>0$ such that, up to isometry, \begin{equation}  \nu= c \mathcal{H}^{3} \res \Sigma ,
\end{equation}
where $\Sigma$ is one of the three following sets 
\begin{enumerate}
\item $\left\lbrace x \; ; \; x_4=0 \right\rbrace \cap \left\lbrace x \; ;\; x_5=0\right\rbrace$, or

\item $\left\lbrace x \; ; \; x_4^2= x_1^2 + x_2^2+ x_3^2 \right\rbrace \cap \left\lbrace x \; ; \; x_5=0\right\rbrace$, or
\item $\left\lbrace x \; ; \; x_4^2= x_1^2 + x_2^2+ x_3^2 \right\rbrace \cap \ \left\lbrace x \; ; \; x_{5}^{2} = 2 x_4^{2} \right\rbrace.$
\end{enumerate}

\end{theorem}

\begin{proof}
First, by Theorem $\ref{necessarycond}$, $\Omega=\supp(\nu) \cap \mathbb{S}^{4}$ is a union of $2p$ distance symmetric $2$-spheres of same radius $r=\sqrt{\frac{1}{2p}}$, and there exists a linear $3$-plane $V$ such that for every sphere $S_i$ of $\Omega$, $S_i \subset V +  \xi_{i}$ where $\xi_i$ is the center of $S_i$ and $|\xi_i|=\sqrt{1-r^2}$. Moreover, for every $i$, $\xi_{i} \in V^{\perp}$. Since $V$ is $3$-dimensional, we can assume without loss of generality that $\left\lbrace \xi_{i} \right\rbrace_{i} \subset \RR^{2} $.
We want to prove that $p=2$ unless $\supp(\nu)$ is the Preiss cone.
By Lemma $\ref{centersunifdist}$, if  $\sigma$ is the spherical component of a conical $3$-uniform measure, and $\Omega$ is its support, then the centers of the $2$-spheres in $\Omega$ are the support of a discrete uniformly distributed measure on $\RR^{2}$, supported on $t \mathbb{S}^{1}$ where $t=\sqrt{1-r^2}$. 
By Proposition  (2.4) in $\cite{KiP}$ where planar uniformly distributed measures with compact support are classified , these centers are either the vertices of a regular $n$-gon or the vertices of $2$ regular $n$-gons of same center and same radius. The fact that, in the definition of distance symmetric points, for a fixed $i$, $\xi_i$ cannot be equidistant to two other centers implies that the centers are either two antipodal points or two pairs of antipodal points.
The first case reduces to the cone $\eqref{KPcone}$.
Indeed, up to isometry, we can take the two centers to be $c_1= (0,0,0,\frac{1}{\sqrt{2}},0)$ and $c_2= -c_1$ since $r=\frac{1}{\sqrt{2}}$ implies that $|c_1 - c_2|=\sqrt{2}$.
Then, taking the sphere $S_1$ to be:
\begin{equation}
S_1= \left\lbrace (z_1,z_2,z_3,0,0)+c_1 \; ; \; {z_1}^{2} + {z_2}^{2} + {z_3}^{2}=\frac{1}{2} \right\rbrace,
\end{equation}
$S_2=-S_1$ and $\Omega= S_1 \cup S_2$, it is easily seen that $\Omega$ is the spherical component of the cone in $\eqref{KPcone} $.

As for the second case, we have $r=\frac{1}{2}$ and we get a rectangle with width $1$ and length $\sqrt{2}$.
Viewing the plane as embedded in $\RR^{5}$ we can find the equation for the support of $\nu$, up to isometry, in the following manner.
Choose the centers of the $4$ $2$-spheres $\left\lbrace S_l \right\rbrace_{l=1}^{4}$, each of which has radius $\frac{1}{2}$, to be $c_1= \left(0,0,0,\frac{1}{2},\frac{\sqrt{2}}{2}\right)$, $c_2=\left(0,0,0,-\frac{1}{2}, \frac{\sqrt{2}}{2} \right)$, $c_3=-c_1$ and $c_4=-c_2$. One can easily verify that $|c_1-c_{2}|=1$, the line passing through $c_1$ and $c_2$ is parallel to the line passing through $c_3$ and $c_4$ and that these two lines are at a distance $\sqrt{2}$ of each other.
Moreover suppose the sphere $S_l$ is described by 
$$ S_l = \left\lbrace (z_1, z_2, z_3,0,0)+ c_l \; , \; {z_1}^2 + {z_2}^2+ {z_3}^2= \frac{1}{4}. \right\rbrace$$

Note that for $z \in \cup _{l=1}^{4}S_l$ if and only if:

\begin{enumerate}

\item $z_1^{2} + z_2^{2}+z_3^{2}= \frac{1}{4}$
\item  $z_4^{2}= \frac{1}{4}$
\item $z_5^{2}= \frac{1}{2}$

\end{enumerate}

Taking the cone over $\cup _{l=1}^{4}S_l$ gives the set:

\begin{equation}\nonumber
\Sigma= \left\lbrace z_1^{2} + z_2^{2}+z_3^{2}=z_4^{2} \right\rbrace \cap \left\lbrace z_5^{2}=2 z_4^{2}\right\rbrace.
\end{equation}

Then $\nu$ is given by \begin{equation}\nonumber
\nu= c \mathcal{H}^{3} \res \Sigma
\end{equation}
for some $c>0$.
\end{proof}
\section{A family of 3-uniform measures}
In the following lemma, we construct a family of distance symmetric points in arbitrary dimension.

\begin{lemma} \label{parallelotope1}
Let $r=2^{-\frac{n+1}{2}}$, $n=0,1,\ldots $. Construct the rectangular parallelotope $R_{n+1}$ in $\mathbb{R}^{n+1}$ inductively in the following manner.
Let $\alpha_1$ be the origin and $\alpha_{2}$ be any point such that $|\alpha_2|=2r$.
Assume the rectangular parallelotope $R_k$ with vertices $\alpha_1, \ldots, \alpha_{2^k}$ has been constructed and is contained in an affine $k$-plane $L_k$. Let $\gamma_k$ be a vector normal to $L_k$ such that $|\gamma_k|=2\sqrt{2^{k}}r$. Set $\alpha_{2^k + i}= \alpha_{i}+\gamma_k$ for $i=1,\ldots,2^{k}$.

Then the vertices of $R_{n+1}$ are  distance symmetric and translating $R_{n+1}$, we can assume that its vertices are contained in $\partial B(0,t)$ where $t=\sqrt{1-r^2}$.
\end{lemma}

\begin{proof}
Let $l_1$ be the permutation of $1$ and $2$. Moreover, let $L_1$ be the line passing through $\alpha_1$ and $\alpha_2$.
We construct the set of permutations $\mathcal{L}$  inductively. Assume that for $1 \leq k \leq n$, we have constructed  $2^k$ permutations $\left\lbrace l_{i} \right\rbrace_{i=1}^{2^{k}}$ satisfying conditions $(1)$ to $(4)$ of Definition $\ref{layers}$ for the first $2^{k}$ indices. 
We first define the action of the first $2^{k}$ layering  
layering functions on the remaining indices in the following manner: for $i \in \left\lbrace 1,2, \ldots, 2^{k}\right\rbrace$, define 
\begin{equation}
l_i(2^{k}+m)=2^{k}+l_i(m),  \mbox{ for } m \in \left\lbrace 1,2, \ldots, 2^{k}\right\rbrace.
\end{equation}
Since $|\alpha_{2^{k}+l_i(m)}-\alpha_{2^{k}+m}|=|\alpha_{m} - \alpha_{l_i(m)}|= 2 \sqrt{i}r $, then $\eqref{welllayered6}$ from Definition $\ref{layers}$ is satisfied.
Moreover, ${l_i}^{2}(2^{k} + m ) = l_{i} (2^{k} + l_{i} (m)) = 2^{k} + {l_{i}}^{2}(m)=2^{k} + m$. So $l_{i}^{-1} = l_{i}$.

We also claim that $l_i$ is bijective on $\left\lbrace 1,2, \ldots, 2^{k+1}\right\rbrace$. 
This follows from the bijectivity on $\left\lbrace 1, \ldots, 2^{k} \right\rbrace$.

It is clear that $(1)$ and $(2)$ from Definition $\ref{layers}$ are satisfied.
To see that $(3)$ is satisfied, suppose there exists $j$, $i$ and $p$ such that $l_i(2^{k}+j)=l_{p}(2^{k}+j)$. Then $2^{k}+l_{i}(j) = 2^{k}+l_{p}(j)$ implying that $i=m$.
Since $l_{i}$ is bijective on $\left\lbrace 1, \ldots, 2^{k} \right\rbrace$ by definition this proves the claim.

We now define the $2^k$ new layering functions. First define $l_{2^{k}}$ in the following manner: 
$$ l_{2^k}(m)=\begin{cases} m+2^{k} & \mbox{if } m \in \left\lbrace 1,2, \ldots, 2^{k} \right\rbrace  \\ m-2^{k} & \mbox{if } m \in \left\lbrace 2^{k}+1, 2^{k}+2, \ldots, 2^{k+1}\right\rbrace \end{cases}$$

Clearly, $l_{2^{k}}$ is a permutation, and satisfies $(2)$, $(4)$ and $\eqref{welllayered6}$ from Definition $\ref{layers}$.

If $i< 2^{k}$, the fact that $l_{i}|_J $ and $l_{2^{k}}|_{J}$ have disjoint ranges for $J=\left\lbrace 1, \ldots, 2^k \right\rbrace$ or $J=\left\lbrace 2^{k}+ 1, \ldots, 2^{k+1} \right\rbrace$ proves that $l_{i}(j) \neq  l_{2^{k}}(j)$ for all $j$.

We now define the remaining layering functions $\left\lbrace l_{2^k+i} \right\rbrace$ in the following manner:

$$ l_{2^k+i}(m)=\begin{cases} 2^{k}+l_{i}(m) & \mbox{ if }m \in \left\lbrace 1,2, \ldots, 2^{k} \right\rbrace  \\ l_{i}^{-1}(m-2^{k}) & \mbox{if } m \in \left\lbrace 2^{k}+1, 2^{k}+2, \ldots, 2^{k+1}\right\rbrace \end{cases}$$

If $m, i \in \left\lbrace 1,2, \ldots, 2^{k} \right\rbrace$, \begin{align*}
|\alpha_{m}-\alpha_{2^{k}+l_i(m)}|^2 
&= |\alpha_{m} - \alpha_{l_{i}(m)} - \gamma_{k}|^{2}, \\
&=  |\alpha_{m} - \alpha_{l_{i}(m)}|^{2} + |\gamma_{k}|^{2}, \mbox{ since } \gamma_{k} \perp (\alpha_{m} - \alpha_{l_{i}(m)}) , \\
&=4ir^{2} + 4. 2^{k} r^2, \\
&= 4(2^{k}+i)r^2.
\end{align*}
So $\eqref{welllayered6}$ is satisfied.
We claim that $(3)$ is also satisfied. Indeed, on one hand, if $i< 2^{k}$, $p< 2^{k}$, the fact that $l_{i}(j) \neq l_{2^{k}+p}(j)$ for all $j$ follows similarly as for $l_{2^{k}}$ and $l_{i}$. On the other hand, 
$$ l_{2^{k}+i}(j)=l_{2^{k}}(j) \implies 2^{k}+l_i(j)=2^{k}+j \implies j=0.$$ 

It is easily seen that, $$l_{2^{k}+i}\left(\left\lbrace 1,2, \ldots, 2^{k} \right\rbrace\right) =\left\lbrace 2^{k}+1,2^{k}+2, \ldots, 2^{k+1} \right\rbrace $$ and $$ l_{2^{k}+i}\left(\left\lbrace 2^{k}+1,2^{k}+2, \ldots, 2^{k+1} \right\rbrace\right)= \left\lbrace 1,2, \ldots, 2^{k} \right\rbrace,$$ by definition of $l_{2^{k}+i}$ and the bijectivity of $l_i$ (and $l_i^{-1}$) on its domain. Therefore, $l_{2^{k}+i}$ is bijective.

Finally for  $j \leq 2^{k}$, $${l_{2^k+i}}^{2}(j)=l_{2^k+i} (2^k + l_{i}(j))=2^{k}+{l_{i}}^{2}(m)=2^{k}+m .$$ A similar argument shows that ${l_{2^k+i}}^{2}(j)=j$ if $j > 2^{k}$. Therefore $(4)$ is satisfied.

This proves that $\left\lbrace \alpha_{1}, \ldots, \alpha_{2^{k+1}} \right\rbrace$ is a distance symmetric set.

Note that $\left\lbrace \alpha_{i} \right\rbrace_{i=1}^{2^{k+1}}$ are the vertices of a rectangular parallelotope contained in the $(k+1)$-affine space $L_{k+1}$ spanned by $L_k$ and $\gamma_k$. This parallelotope has $R_k$ as one of its faces, all the edges of $R_k$, $\left\lbrace [ \alpha_{i+2^{k}} \alpha_{j+2^k}] \right\rbrace_{(i,j)}$ and $\left\lbrace [ \alpha_{i} \alpha_{i+2^k}] \right\rbrace_{1 \leq i \leq 2^k}$ as its edges.
Moreover, the main diagonal of $R_{k+1}$ has length $|\alpha_1 - \alpha_{2^{k+1}}| = 2 \sqrt{2^{k+1}-1}r$.
By induction, repeating this process for $k=n$, we get $2^{n+1}$ points forming a rectangular parallelotope $R_{n+1}$ in $\mathbb{R}^{n+1}$ with main diagonal having length $2 \sqrt{2^{n+1}-1}r=2t$. This implies that $R_{n+1}$ is inscribed in a sphere of radius $t$.
By translating, we can assume that $R_{n+1}$ is inscribed in $\partial B_{t}(0)$.

Finally, note that $|\alpha_{i} - \alpha_{l_{2^{n+1}-1}(i)}|=2t$. Since $\alpha_{i}$ and $ \alpha_{l_{2^{n+1}-1}(i)}$ are in $\partial B_{t}(0)$ they must be antipodal points.
Therefore, $\alpha_{i}=-\alpha_{l_{2^{n+1}-1}(i)}$
\end{proof}

This allows us to construct support $2$-uniform measures in $\RR^{n+4}$, for any integer $n$. More precisely,

\begin{corollary}\label{familyoflocally2unif}
Let $n \geq 1$, $r=\frac{1}{2^{n+1}}$, $t=\sqrt{1-r^2}$, and $\left\lbrace \alpha_1 \ldots, \alpha_{2^{n+1}} \right\rbrace$ be a distance symmetric set as in Lemma $\ref{parallelotope1}$, such that $|\alpha_{j}|=t$, for $j=1, \ldots 2^{n+1} $.
Define the points $c_i$ in $\RR^{n+4}$ to be 
$$c_i= (0,0,0,\alpha_{i})$$ and the corresponding $2$-spheres $S_i$ as:
\begin{equation}\label{dyadicspheres}
S_{i}=\left\lbrace z \in \RR^{n+4}; z=(z_1,z_2,z_3,\alpha_{i}), {z_1}^2+{z_2}^{2}+{z_3}^{2} = r^2. \right\rbrace
\end{equation}
In particular, for each $i$, $S_i \subset \mathbb{S}^{n+3}$.
Let $\Omega$ be the set 
\begin{equation}\label{dyadicsupport}
\Omega= \left(\bigcup_{i=1}^{2^{n+1}} S_i \right),
\end{equation}
and $\sigma$ the measure 
\begin{equation}\label{definitionsigma}
\sigma= \mathcal{H}^{2} \res \Omega.
\end{equation}
Then for all $x \in \Omega$, for $r \leq 2$, we have:
\begin{equation}\label{sigma2uniform}
\sigma(B(x,r))=\pi r^{2}.
\end{equation}

\end{corollary}
\begin{proof}
This follows directly from the fact that the $c_j$ are distance symmetric and Theorem $\ref{suffcond} $.

\end{proof}

Using Corollary $\ref{cond3unifconical}$ we obtain the following lemma.
\begin{lemma}\label{3unifexample}
Let $R_{n+1}$ be the parallelotope from Lemma $\ref{parallelotope1}$, $n \geq 0$. For every $l=1, \ldots, 2^{n+1}$ set the point $c_l \in \mathbb{R}^{n+4}$ to be:
\begin{equation}
c_l = \left(0,0,0, \alpha_{l} \right).
\end{equation} 
Let $V$ be a linear $3$-plane in $\mathbb{R}^{n+4}$, $S_l$ be the $2$-sphere 
$$S_l= (V+c_l) \cap \mathbb{S}^{n+3},$$
centered at $c_l$, $\Omega$ be the set 
\begin{equation}
\Omega= \bigcup_{l=1}^{2^{n+1}} S_l,
\end{equation}
and $\Sigma$ be the set  \begin{equation}
\Sigma= \left\lbrace x \in \mathbb{R}^{n+4} ; \frac{x}{|x|} \in \Omega \right\rbrace \bigcup \left\lbrace 0 \right\rbrace .
\end{equation}
Then $\nu= \mathcal{H}^{3} \res \Sigma$ is a $3$-uniform measure and for any $x \in \Sigma$, $r>0$,
\begin{equation}
\nu(B(x,r))=\frac{4}{3} \pi r^{3},
\end{equation}

\end{lemma} 
\begin{proof}
This is a direct consequence of Corollary $\ref{familyoflocally2unif}$ and Lemma $\ref{cond3unifconical}$.
\end{proof}

\begin{remark}
Take the origin in $\mathbb{R}^{n+1}$ to be the center of symmetry of the corresponding parallelotope from Theorem $\ref{parallelotope1}$, and choose an orthonormal basis for $\mathbb{R}^{n+1}$  in the following way.
Let $e_1$ be parallel to the side of length $2 \sqrt{1}r$, $e_2$ parallel to the side of length $2 \sqrt{2}r$, $e_3$  parallel to the side of length $2 \sqrt{4}r$, $\ldots$, $e_{n+1}$  parallel to the side of length $2 \sqrt{2^{n}}r$.
It is easy to verify that in this orthonormal system, the parallelotope $R_{n+1}$ is given by the equations $x_l^2= 2^{l-1}r^2$ for $l=1, \ldots, n+1$.

\end{remark}
Using the remark, Lemma $\ref{parallelotope1}$, Theorem $\ref{3unifexample}$ can be reformulated in the following way.
\begin{theorem}
For every $k=0,1, \ldots$, let $C_k$ be the  cone in $\mathbb{R}^{k+4}$ consisting of the points $x=(x_1, \ldots, x_{k+4})$ satisfying
\begin{equation*}
x \in \left\lbrace x_4^2= x_1^2 + x_2^2+ x_3^2 \right\rbrace \cap \bigcap_{l=1}^{k} \left\lbrace x_{l+4}^{2} = 2^{l} x_4^{2} \right\rbrace.
\end{equation*}
Then, for all $x \in C_k$, for all $r>0$
\begin{equation*}
\mathcal{H}^{3}(B(x,r) \cap C_k )= \frac{4}{3} \pi r^{3}.
\end{equation*}
\end{theorem}

\section{Construction of distance symmetric points}
Our aim now is to find a systematic way of producing layerings. To do this we need to define a graph associated to each layering and find conditions on the graph guaranteeing its embeddability in $\RR^{d}$.

\begin{definition}
Let  $\mathcal{L}$ be a layering. We define its graph $G_{\mathcal{L}}$  to be the weighted graph composed of
\begin{enumerate}
\item the vertices $V(G)= \left\lbrace v_{i} \right\rbrace_{2p}$
\item the edges $E(G)= \left\lbrace \left\lbrace v_{i}, v_{j} \right\rbrace \right\rbrace_{1\leq i < j \leq 2p}$
\item the weight $w \left(\left\lbrace v_{i}, v_{j} \right\rbrace \right)=d_{ij}$ where $d_{ij}$ are the distance functions that arise from $\mathcal{L}_{\nu}$.
\end{enumerate}

\end{definition}

We start by proving the simple observation that a  layering $\left\lbrace l_i \right\rbrace_{i=1}^{m}$ consists of an edge-coloring of the complete graph.

\begin{proposition}
Let $G=K_{2p}$ be the complete graph on the vertices $\left\lbrace v_i \right\rbrace _{i=1}^{2p}$. 

Then $\mathcal{L}=\left\lbrace l_i \right\rbrace _{i=0}^{m}$ is a layering if and only if the assignment  $c: E(K_{2p})  \to \left\lbrace 1, \ldots, 2p-1 \right\rbrace$ of colors defined by $c(\left\lbrace v_i, v_j \right\rbrace  )=d_{ij}$ is a $(2p-1)$ coloring of the edges of $K_{2p}$. We call  $G_{\mathcal{L}}$ the graph associated to the layering.

Moreover, if there exist numbers $d_{ij}$, for $1 \leq i  < j \leq 2p$ such that $d_{ij} \in \left\lbrace 1, \ldots, 2p-1 \right\rbrace$ for all $i,j$ and the assignment $c({v_i ,v_j})=d_{ij}$ defines a $(2p-1)$ edge-coloring of $G$, define the functions $\mathcal{L}=\left\lbrace  l_{k} \right\rbrace_{i=0}^{2p-1}$ in the following manner:

\begin{itemize}
\item $l_{0}(j)=j$ for all $j \in \left\lbrace 1, \ldots, 2p-1 \right\rbrace$,
\item $l_{k}(i)=j$ for $k>0$, where $j$ is the integer such that $d_{ij}=k$.
\end{itemize}
 Then, up to relabeling of the vertices, $\mathcal{L}$ is a layering.
\end{proposition}

\begin{proof}
The proof follows directly from the definition of a layering.
Indeed, for $c$ to define an edge-coloring, we only need to prove that $d_{ij}=d_{ik}$ implies that $j=k$. Clearly,if  $d_{ij}=d_{ik}$ then  $l_{d_{ij}}(i)=j$ and $l_{d_{ij}}(i)=k=j$.

Conversely, if $G$ is as described, we first prove that the functions $l_k$ are well-defined bijections.
Pick any $k>0$ and $1 \leq i   \leq 2p$. Since $v_{i}$ is adjacent to $2p-1$ edges, and $c$ is a $(2p-1)$ coloring, there exists a unique $j$ such that $d_{ij}=k$.

We can relabel the vertices so that $l_{k}(1)=k+1$.
The fact that $l_{k}^{-1}=l_k$ is a consequence of the fact that $d_{ij}=d_{ji}$. Finally, suppose that there exists $j$ such $l_{k}(j)=l_{k'}(j)=i$. Then $k=k'=d_{ij}$. This ends the proof.

\end{proof}

We now wish to get results in the other direction. In other words, if a weighted graph is given, what conditions will guarantee that there exists a $3$-uniform conical measure associated to it? More precisely, by defining the weighted graph $G$ associated to a $(2p-1)$-coloring of $K_{2p}$ (which assigns to each edge the weight $c(\left\lbrace v_{i}, v_{j} \right\rbrace )=d_{ij}$), what conditions on $G$  guarantee the existence of a conical $3$-uniform measure $\nu$ such that $G=G_{\nu}$?
By Theorem $\ref{suffcond}$, every set of $2p$ distance symmetric points for $r=\sqrt{2p}$ gives rise to a $3$-uniform measure. We will use $\ref{Bl}$ to find conditions on a set of distances $d_{ij}$ associated to a layering that guarantee its embeddability in Euclidean space.
\begin{definition}\label{delta}
Let $\mathcal{L}=\left\lbrace l_{i} \right\rbrace_{i=0}^{2p-1}$ be a layering. We define the matrix $\Delta_{\mathcal{L}}$ associated to the layering to be 
$$ (\Delta_{\mathcal{L}})_{ij} = \frac{2p-1-2d_{ij}}{2p-1}.$$
\end{definition}

\begin{theorem}\label{spectralgap}
Let $p \in \NN$, $\mathcal{L}=\left\lbrace l_i \right\rbrace_{i=0}^{2p-1}$ be a layering, $r=\frac{1}{\sqrt{2p}}$ and $t=\sqrt{1-r^2}=\sqrt{\frac{2p-1}{2p}} $. Then there exists an distance symmetric set of $2p$ points $\left\lbrace \xi_{i} \right\rbrace_{i=1}^{2p}$ in $t \mathbb{S}^{2p-2}$ if and only if the spectral gap $\lambda_{G}$ of the Laplacian of the graph $G_{\mathcal{L}}$ associated to the layering satisfies: \begin{equation}\label{spectralgapineq}
\lambda_{G} \geq p(2p-1),
\end{equation}

\end{theorem}
\begin{proof}
By Theorem $\ref{Bl}$, if we take our semi-metric space to be $\left\lbrace \xi_{i} \right\rbrace_{i=1}^{2p}$ with the distance set $\left\lbrace  \frac{\sqrt{2p-1}}{\sqrt{2p}} \arccos\left(\frac{2p-1-2d_{ij}}{2p-1}\right)  \right\rbrace $, there exist points $\left\lbrace \xi_{i} \right\rbrace_{i=1}^{2p} \subset \mathbb{R}^{2p-1}$, $|\xi_{i}|=t$ with distance set $|{\xi}_i - {\xi}_j|_{t\mathbb{S}}= \frac{\sqrt{2p-1}}{\sqrt{2p}} \arccos\left(\frac{2p-1-2d_{ij}}{2p-1}\right)=\overline{d}_{ij}$ if and only if the matrix $\Delta$ given by:

\begin{equation}
\Delta_{ij}=cos \left( \frac{\overline{d}_{ij}}{t} \right)=\frac{2p-1-2d_{ij}}{2p-1}
\end{equation}
is positive semi-definite.

Note that for this choice of $\overline{d}_{ij}$, if we find points $\left\lbrace \xi_{i} \right\rbrace_{i=1}^{2p}$ with the prescribed distance set, their euclidean distance will be:

\begin{align*}
|\xi_{i} - \xi_{j}|^{2} &=|\xi_{i}|^{2} + |\xi_{j}|^2 - 2 \left\langle \xi_{i}, \xi_{j} \right\rangle, \\ &=2 \; . \; \frac{2p-1}{2p}- 2\; . \; \frac{2p-1}{2p} cos \left(  \frac{\overline{d}_{ij}}{t} \right), \\ &=2 \; . \; \frac{2p-1}{2p}- 2 \; . \; \frac{2p-1}{2p} \;. \; \frac{2p-1-d_{ij}}{2p-1}, \\ &=4 \;.\; d_{ij} \;.\; \frac{1}{2p}.
\end{align*}

We will first rewrite the matrix $\Delta$ in terms of the Laplacian of $G$ and the fact that $\Delta$ is positive semi-definite will then allow us to deduce the lower bound on $\lambda_{G}$.
Denote the Laplacian of $G$ by $L$.
For $i \neq j$, 
\begin{equation}
\Delta_{ij}=1- \frac{2}{2p-1} d_{ij}=1+\frac{2}{2p-1} L_{ij}
\end{equation}
and for $i=j$,
\begin{equation}
\Delta_{ii}=1=1+\frac{2}{2p-1} \; . \; \frac{2p(2p-1)}{2} -2p=1 +\frac{2}{2p-1} L_{ii} -2p.
\end{equation}
Therefore, $$ \Delta_{ij}=1-2p\delta_{ij}+ \frac{2}{2p-1} L_{ij},$$
where $\delta_{ij}$ is the Kronecker symbol.
This follows from the fact that each vertex of $G$ has degree $\frac{2p(2p-1)}{2}$. Indeed, each $v_i$ has $2p-1$ edges adjacent to it, all of distinct weight between $1$ and $2p-1$. So $d(v_i)= \sum_{i=1}^{2p-1} i = \frac{2p(2p-1)}{2}$.

This implies that 
\begin{equation}
\Delta=J-2p I_{2p} + \frac{2}{2p-1} L,
\end{equation}
where $J$ is the matrix whose entries are all $1$ and $I_{2p}$ is the identity matrix.

$J$ has eigenvalues $2p$ and $0$, the vector $e_{1}= (1, \ldots , 1)$ is a common eigenvector of $J$ for the eigenvalue $2p$ and of $L$ for the eigenvalue $0$. Hence we can choose $e_1$ to be a  common eigenvector corresponding to the $0$ eigenvalue for $L$.
Let $e$ be an eigenvector of $L$ orthogonal to $e_1$ and $\lambda$ the corresponding eigenvalue. Since $e$ is orthogonal to $e_1$, 
\begin{align*}
\Delta \; . \; e &= J.e- 2p  \; e + \lambda \frac{2}{2p-1} \; e,\\
&= \left( \frac{2 \lambda}{2p-1}-2p \right) \; e.
\end{align*}

Hence, $\Delta$ is positive semi-definite if and only if $\frac{2 \lambda}{2p-1}-2p \geq 0$ if and only if $\lambda \geq p(2p-1)$.
In particular, if $\lambda_G$ is the second smallest eigenvalue of $L$, $\Delta$ is positive semi-definite if and only if $\lambda_{G} \geq p(2p-1)$.

\end{proof}

The fact that the matrix $\Delta$ from the proof of Theorem $\ref{charactergeom}$ is positive semi-definite encodes information on the geometry of the set of points it describes.
We start with a lemma.

\begin{lemma} \label{cyclerectangle}
If $\left\lbrace l_i \right\rbrace $ is a layering, then for $j=1, \ldots, 2p$ we have:
\begin{enumerate}
\item $l_{2p-1}(j)=2p+1-j$
\item $l_{i} \circ l_{2p-1} = l_{2p-1} \circ l_{i}= l_{2p-1-i}$
\end{enumerate}
\end{lemma}
 \begin{proof}
To prove $(1)$,  note that $|\xi_{j} - \xi_{l_{2p-1}(j)}|=2\sqrt{2p-1}r=2 \frac{\sqrt{2p-1}}{\sqrt{2p}}=2t$. So $\xi_{j}$ and $\xi_{l_{2p-1}(j)}$ are antipodal points.
Now pick $j$. Since $\xi_j$ and $\xi_{l_{2p-1}(j)}$ are antipodal, we have:
$$|\xi_{1} - \xi_{j}|^2 + |\xi_{1} - \xi_{l_{2p-1}(j)}|^{2} = |\xi_{l_{2p-1}(j)} - \xi_{j}|^{2},$$
which implies, after dividing by $4r^2$, that 
$$j-1+ l_{2p-1}(j) -1= 2p-1$$
since $l_{1}(j)=j+1$ for all $j$. This proves $(1)$.
Now to prove $(2)$ , consider the rectangle formed by $\xi_{j}, \xi_{l_{i}(j)}, \xi_{l_{2p-1}(j)}, \xi_{l_{2p-1} \circ l_{i}(j)}$.
We have 
$$ |\xi_{j} - \xi_{l_i(j)}|^2 + |\xi_{j} - \xi_{l_{2p-1} \circ l_i (j)}|^{2}=4(2p-1)r^2.$$
This implies that $i + |\xi_{j} - \xi_{l_{2p-1} \circ l_i (j)}|^{2}=2p-1$ and 
\begin{equation}\label{perm1}
l_{2p-1} \circ l_i = l_{2p-1-i}.
\end{equation}
Applying $l_i$ to the left in $\eqref{perm1}$, we get: 
\begin{equation}\label{perm2}
l_{2p-1}=l_{2p-1-i} \circ l_{i}.
\end{equation}
We obtain the other identities similarly.
\end{proof}
\begin{theorem}\label{constructpoints}
Let $\left\lbrace l_{i} \right\rbrace_{i=0}^{2p-1}$ be a  layering and let $\Delta$ be the matrix $\Delta= J-2pI_{2p} + \frac{2}{2p-1}L $ where $J$ is the matrix with $1$ in all its entries, $I_{2p}$ is the identity matrix and $L$ is the Laplacian of the graph associated to the layering. Then, if $\Delta$ is positive semi definite, there exists a matrix $A$ of rank at most $p$ such that:
\begin{equation}
\Delta= A^{\intercal} A
\end{equation}
and  the columns $\left\lbrace \xi_{i} \right\rbrace_{i=1}^{2p}$ of $A$ form a set of  distance symmetric points in $t\mathbb{S}^{2p-1}$ where $t=\sqrt{\frac{2p-1}{2p}}$. Moreover $p$ must be even.
\end{theorem}

\begin{proof}
Since $\Delta$ is positive semi-definite, there exists a set of  distance symmetric points $\left\lbrace\xi_{i}\right\rbrace$ in $t \mathbb{S}^{2p-2}$ by Theorem $\ref{spectralgap}$.

We prove that $p$ is even.
Consider the sets $A_j=\left\lbrace j , l_{1}(j),\ldots, l_{2p-1}(j), l_{2p-2}(j) \right\rbrace$. We claim that for $j \neq k$, either $A_j = A_k$ or $A_j \cap A_k =\emptyset$.
Suppose that $A_j \cap A_k \neq \emptyset$ and let $s$ be in the intersection. Notice that by Lemma $\ref{cyclerectangle}$ if $s \in A_j \cap A_k$, then $l_{1}(s), l_{2p-1}(s), l_{2p-2}(s)$ are all in $A_j \cap A_k$. Since those elements are all distinct, $A_j = A_k$. Therefore these sets partition $\left\lbrace 1, \ldots, 2p \right\rbrace$ which implies that $4$ divides $2p$ and $p$ is even.

To prove that $\Delta$ has rank at most $p$, we rewrite it in a more convenient way. Let $\left\lbrace e_j \right\rbrace$ be an orthonormal basis of $\RR^{2p}$. Define for each $i= 0, \ldots, 2p-1$ the permutation matrix $A_i$ defined by 
\begin{equation}\label{permmatrix}
A_{i}(e_j)= e_{l_i(j)}.
\end{equation}
We claim that $\Delta$ can be written as:
\begin{equation}\label{rewritedelta}
\Delta = \sum_{i=0}^{p-1} \frac{2p-1-2i}{2p-1} \left( A_i - A_{2p-1-i} \right).
\end{equation}
First note that $\frac{2p-1-2(2p-1-i)}{2p-1} = - \frac{2p-1-2i}{2p-1}$.
Now the image of $e_j$ by the matrix on the right of $\ref{rewritedelta}$ is:

\begin{align*}
\sum_{i=0}^{p-1} \frac{2p-1-2i}{2p-1}  \left( A_{i} . e_j - A_{2p-1-i} . e_j \right) &= \sum_{i=0}^{p-1} \frac{2p-1-2i}{2p-1} e_{l_i(j)} + \frac{2p-1-2(2p-1-i)}{2p-1} e_{l_{2p-1-i}(j)}, \\&=  \sum_{i=0}^{2p-1} \frac{2p-1-2d_{j,l_i(j)}}{2p-1} e_{l_{i}(j)}, \\ &=\sum_{k=0}^{2p-1} \Delta_{jk} e_{k},
\end{align*}
proving the claim.

Consider the orthogonal basis $\left\lbrace u_i \right\rbrace_{i=1}^{2p}$ defined in the following way:

$$u_j= \begin{cases} e_j + e_{2p+1-j} , \; j \leq p \\ e_j- e_{2p+1-j}, \; j \geq p+1. \end{cases}$$
We claim that $\Delta u_j=0$ for $j \leq p$ and $\Delta u_j \in span \left\lbrace u_{p+1}, \ldots, u_{2p} \right\rbrace$ for $j \geq p+1$.

Indeed, for $j \leq p$,

\begin{align*}
\Delta u_j &= \sum_{i=0}^{p-1} \frac{2p-1-2i}{2p-1}  \left( A_i . u_j - A_{2p-1-i}. u_j \right), \\ &= \sum_{i=0}^{p-1} \frac{2p-1-2i}{2p-1}  \left( A_i . e_j+A_i . e_{2p+1-j} - A_{2p-1-i}. e_j - A_{2p-1-i} . e_{2p+1-j}\right), \\ &=\sum_{i=0}^{p-1} \frac{2p-1-2i}{2p-1} \left(e_{l_i(j)} + e_{l_{i}(2p+1-j)} - e_{l_{2p-1-i}(j)}-e_{l_{2p-1-i}(2p+1-j)} \right), \\ &=\sum_{i=0}^{p-1} \frac{2p-1-2i}{2p-1} \left(e_{l_i(j)} + e_{l_{2p-1-i}(j)} - e_{l_{2p-1-i}(j)}-e_{l_{i}(j)} \right), \mbox{ by Lemma } \ref{perm1},\\
&=0
\end{align*}
On the other hand, for $j \geq p+1$:
\begin{align*}
\Delta u_j &= \sum_{i=0}^{p-1} \frac{2p-1-2i}{2p-1}  \left( A_i . u_j - A_{2p-1-i} . u_j \right), \\ &= \sum_{i=0}^{p-1} \frac{2p-1-2i}{2p-1}  \left( A_i  . e_j-A_i . e_{2p+1-j} - A_{2p-1-i}. e_j + A_{2p-1-i} .  e_{2p+1-j}\right), \\ &=\sum_{i=0}^{p-1} \frac{2p-1-2i}{2p-1} \left(e_{l_i(j)} - e_{l_{i}(2p+1-j)} - e_{l_{2p-1-i}(j)}+e_{l_{2p-1-i}(2p+1-j)} \right), \\ &=\sum_{i=0}^{p-1} \frac{2p-1-2i}{2p-1} \left(e_{l_i(j)} - e_{l_{2p-1-i}(j)} - e_{l_{2p-1-i}(j)}+e_{l_{i}(j)} \right), \mbox{ by Lemma } \ref{perm1},\\
&=2\sum_{i=0}^{p-1} \frac{2p-1-2i}{2p-1} u_{ \min(l_i(j), l_{2p-1-i}(j))}
\end{align*}

This proves that $\Delta$ has rank at most $p$.

Finally, we describe how to find the corresponding  distance symmetric points.
Since $\Delta_{ij}= \left\langle \xi_i , \xi_j \right\rangle$ for the  points whose existence is guaranteed by Theorem $\ref{Bl}$, if we find a matrix $A$ with columns $x_i$  such that
$$\Delta =  A^{\intercal} A,$$
then $\Delta_{ij} = \left\langle x_i , x_j \right\rangle$ and we can set $\xi_{i} = x_i$.
To find such a matrix, we diagonalize $\Delta$. Since it is symmetric, there exists an orthogonal matrix $P$ and a diagonal matrix $D$ so that: $\Delta=PDP^{\intercal}$.
Since $\Delta$ is positive semi-definite, all the entries of $D$ are non-negative. Denoting by $D^{\frac{1}{2}}$ the diagonal matrix with entries the square roots of the entries of $D$, we can write:
$$ \Delta= PD^{\frac{1}{2}} D^{\frac{1}{2}} P^{\intercal}.$$
Choose $A^{\intercal}$ to be $PD^{\frac{1}{2}}$.
By Theorem  [7.2.10] in $\cite{HJ}$, $A$ and $\Delta$ have the same rank, which ends the proof.
\end{proof}

We can now put those results together in the following theorem. 

\begin{theorem}
Let $\mathfrak{G}$ be the set of weighted graphs $G$  satisfying:
\begin{itemize}
\item $G=K_{4p}$, $p \in \NN$,
\item $G$ is weighted by $w: E(G) \to \left\lbrace 1, \ldots 4p-1 \right\rbrace$ and the assignment of labels corresponding to $w$ is an edge-coloring of $G$,
\item The second smallest eigenvalue $\lambda_{G}$ of the (non-normalized) Laplacian of $G$ satisfies: $$\lambda_{G} \geq \frac{4p(4p-1)}{2}.$$
\end{itemize}
For every graph $G \in \mathfrak{G}$, $|V(G)|=4p$, let $\mathcal{L}$ be the layering associated to it. Construct the set of points $\left\lbrace \xi_{i} \right\rbrace_{i=1}^{4p} \subset  \mathbb{R}^{4p-1}$ associated to $\mathcal{L}$, set  $ c_{i}= (0,0,0,\xi_{i})$ for $i=1, \ldots, 4p $ and define $S_i$ to be the $2$-sphere of radius $r=\sqrt{\frac{1}{4p}}$ centered at $c_i$, such that $S_i = \left( V + c_i \right) \cap\mathbb{S}^{4p+1}$ where $V= \RR^{3} \times \left\lbrace 0 \right\rbrace$. Setting $\Omega=\bigcup_{i=1}^{4p} S_{i} $ and $$\Sigma= \left\lbrace x \in \RR^{4p+2} \; ; \; \frac{x}{|x|}\in \Omega\right\rbrace \bigcup \left\lbrace 0 \right\rbrace,$$
and $\nu = \mathcal{H}^{3} \res \Sigma$, we have for all $x \in \Sigma$, $r>0$,
$$ \nu(B(x,r))= \frac{4 \pi}{3} r^{3}.$$
In particular, $\nu$ is $3$-uniform.

\end{theorem}
\begin{proof}
The theorem is a direct consequence of Theorems $\ref{charactergeom}$, $\ref{constructpoints}$ and $\ref{spectralgap}$.
\end{proof}

\begin{example}
Consider the graph whose adjacency matrix is given by
\begin{equation}  
 \begin{bmatrix} 0 & 1 & 2 & 3 & 4 & 5 & 6 & 7 \\
                        1 & 0 & 3 & 5 & 2 & 4 & 7 & 6 \\
                        2 & 3 & 0 & 1 & 6 & 7 & 4 & 5 \\
                        3 & 5 & 1 & 0 & 7 & 6 & 2 & 4\\
                        4 & 5 & 6 & 7 & 0 & 1 & 2 & 3\\
                        2 & 4 & 7 & 6 & 1 & 0 & 3 & 5 \\
                        6 & 7 & 4 & 5 & 2 & 3 & 0 & 1\\
                        7 & 6 & 2 & 4 & 3 & 5 & 1 & 0 \end{bmatrix} 
\end{equation}
One can easily verify that the entries $d_{ij}$ of the matrix $A$ give an edge coloring of $K_8$. Moreover the matrix $7 \Delta$ , where $\Delta$ is as in Definition $\ref{delta}$, is given by:

\end{example}
\begin{equation}  
 \begin{bmatrix}7 & 5 & 3 & 1 & -1 & -3 & -5 & -7 \\
                        5 & 7 & 1 & -3 & 3 & -1 & -7 & -5 \\
                        3 & 1 & 7 & 5 & -5 & -7 & -1 & -3 \\
                        1 &-3 & 5 & 7& -7 & -5 & 3 & -1\\
                        -1 & -3 & -5 & -7 & 7 & 5 & 3 & 1\\
                         3 & -1 & -7 & -5& 5 & 7& 1 & -3 \\
                        -5 & -7 & -1 & -3&  3 & 1 & 7 & 5 \\
                      -7 & -5 & 3 & -1 & 1 &-3 & 5 & 7 \end{bmatrix} 
\end{equation}
which has 4 positive eigenvalues and 4 null eigenvalues. This means that there exist points $\left\lbrace \xi_{i} \right\rbrace_{i=1}^{8}$ with prescribed distances $| \xi_{i} - \xi_{j} | =  \sqrt{\frac{d_{ij}}{2}}$, embedded in the sphere $t \mathbb{S}^{3}$ where $t=\sqrt{\frac{7}{8}}$, which are distance symmetric.
This configuration is composed of the tetrahedron given by $\left\lbrace \xi_{i} \right\rbrace_{i=1}^{4}$ and its antipode on $t\mathbb{S}^{3}$. Of course, one can construct a $3$-uniform measure in $ \RR^{7}$ using these points as in Theorem $\ref{suffcond}$.

\newpage
\section*{Appendix}
In the following appendix, we prove Lemma $\ref{umbilic}$. The proof is similar to the the proof of Theorem $\ref{KoP}$ with the added difficulty of the measure being in a space of higher codimension.

 Let us start by restating the lemma.

\begin{lemma}
Let $\mu$ be a $3$-uniform measure in $\RR^{d}$, $\sigma$ its spherical component and $\supp(\sigma)= \Omega$. Then:

$$ \mathcal{R} \subset \bigcup_{\alpha} S_{\alpha},$$
where the $S_{\alpha}$'s are $2$-spheres and $\mathcal{R}$ is the regular part of $\Omega$ as defined in Theorem $\ref{uniformlydistmeasure}$.
\end{lemma}

We divide the proof of this lemma into claims which will be proven separately. The setting of the claims is the following: we pick $Q \in \mathcal{R}$. Without loss of generality, by rotating and translating  $\Omega$, we can assume that $Q=0$ and $\Omega \subset \partial B(-p, 0)$ where $p=(0,0,1,0, \ldots, 0)$.
We can choose a basis $\left\lbrace e_1, e_2 \right\rbrace$ of $P=T_{0} \Omega$ satisfying the following: in a neighborhood $U$ of $0$, writing $\overline{x}$ the projection of $x$ on $P$, there exist $d-2$ real analytic functions $z_i$ of $\overline{x}$ so that

\begin{equation}
\Omega \cap U = \left\lbrace \overline{x}+\sum_{i=3}^{d} z_{i}(\overline{x}) e_i ; \overline{x} \in P \cap U \right\rbrace, 
\end{equation}
and such that $z_i(0) =0$ , $\nabla z_i(0)=0$ for all $i$ and $\nabla^{2} z_{4}(0) = diag( \lambda_{1}, \lambda_{2})$.
\begin{claim}\label{umbilic1}
$$\nabla^{2} z_{3}(0)=  \begin{bmatrix}
    -1 & 0 \\
    0 & -1 
  \end{bmatrix}
.$$
\end{claim}
\begin{proof}
Indeed since $\Omega \subset \mathbb{S}^{d-1} - p$, 
\begin{equation}\label{sphere}
x_1^2+x_2^2+(z_3 +1)^2+z_4^2+ \ldots + z_d^{2}=1.
\end{equation}
Differentiating $\eqref{sphere}$ with respect to $x_1$ then $x_2$, and plugging in $ z_{i}(0)=0$ and $\nabla z_{i}(0)=0$, we get:
\begin{equation} \partial_{2} \partial_{1} z_{3} (0) (z_3(0)+1)=0 ,
\end{equation}
and hence $\partial_{2} \partial_{1} z_{3} (0)= \partial_{1} \partial_{2} z_{3} (0)=0$.
Differentiating $\eqref{sphere}$ twice with respect to $x_1$ and plugging in $z_{i}(0)=0$ and $\nabla z_{i}(0)=0$, we get:
\begin{equation}
1+ (z_3(0)+1) \partial_{1} \partial_{1} z_{3}(0) =0
\end{equation}
and hence $\partial_{1} \partial_{1} z_{3}(0)=-1$.
Similarly, we get $\partial_{2} \partial_{2} z_{3}(0)=-1$.
\end{proof}
We now write for every $j \geq 5$ \begin{equation} \label{Hessian}
\nabla^{2}z_j(0)= \begin{bmatrix}
    \mu_{1,j} & m_j \\
    m_j & \mu_{2,j}
  \end{bmatrix}
  \end{equation}
  Denoting by $\rho= \sqrt{x_1^2 + x_2^2}$, we can write the following Taylor expansions for the $z_j$'s:
  \begin{align}
  z_3  &=- \frac{1}{2} \rho^{2} + O(\rho^3) \label{Taylor1}, \\
  z_4&= \frac{1}{2} (\lambda_{1} x_1^2 + \lambda_2 x_2^2) + O(\rho^3), \label{Taylor2}, \\
  z_j &= \frac{1}{2} ( \mu_{1,j} x_1^2 + \mu_{2,j} x_2^2 + 2m_j x_1 x_2) + O(\rho^3). \label{Taylor3}
  \end{align}
We will first use the area formula to write a Taylor expansion for $\mathcal{H}^{2}(B(0,r) \cap \Omega)$ for $r$ small in terms of the  $\lambda_{j}$'s , $\mu_{i,j}$'s  and $m_j$'s.
We then use the fact that $\mathcal{H}^{2} \res \Omega$ is locally $2$-uniform to establish relations between the $\lambda_{j}$'s, $\mu_{i,j}$'s and $m_j$'s.
We start by writing the integrand $\mathcal{D}$ appearing in the area formula in terms of the  the  $\lambda_{j}$'s , $\mu_{i,j}$'s  and $m_j$'s.
  \begin{claim}\label{umbilic2}
  For $\overline{x}=(x_1,x_2) \in P \cap U$, we have:
  \begin{equation}\label{D(x)}
  \mathcal{D}(\overline{x})=1+\alpha x_1^2 + \beta x_2^2 + \gamma x_1 x_2 + O(\rho^4),
  \end{equation}
  where
  \begin{align}
  \alpha &= 1+\lambda_1^2 + \sum_{j} (\mu_{1,j}^2 + m_j^2),\\
  \beta &=1+ \lambda_2^2+ \sum_{j} (\mu_{2,j}^2 + m_j^2), \\
  \gamma &= \sum_{j} 2 m_j (\mu_{1,j} + \mu_{2,j}).
  \end{align}
  Moreover, if we write $x_1= \rho a_1$ and $x_2= \rho a_2$ where $a_1=a_1(\theta)=cos(\theta)$ and $a_2= a_2(\theta)=sin(\theta)$, then $\eqref{D(x)}$  becomes:
  \begin{equation}\label{polarD(x)}
  \mathcal{D}(\rho, \theta)= 1+ \overline{B}(\theta) \rho^2 + O(\rho^4)
  \end{equation}
  where  $\overline{B}(\theta)= \alpha a_1^2+ \beta a_2^2 + \gamma a_1 a_2$.
  \end{claim}
  \begin{proof}
 $\mathcal{D}$ is the sum of the squares of all $2 \times 2$ minors of the matrix $Jz(\overline{x})$ which is given (up to a term $O(\rho^3)$ in each entry)  by:
  \begin{equation} \label{Jz}
  \begin{bmatrix}
   1 & 0\\
  0 & 1\\
  -x_1 & -x_2 \\
  \lambda_{1} x_1 & \lambda_2 x_2 \\
  \mu_{1,4} x_1 + m_4 x_2 & \mu_{2,4} x_2 + m_4 x_1\\
  \vdots & \vdots \\
  \mu_{1,d} x_1 + m_d x_2 & \mu_{2,d} x_2 + m_d x_1
 
  \end{bmatrix}
  \end{equation}

  If we denote by $\tau$ the permutation of $1$ and $2$ then:
  \begin{align*}
  \mathcal{D}(\overline{x}) = & 1 + \sum_{i=1}^{2} x_i^2 + \lambda_{i}^2 x_{i}^2 +\left(\sum_{i=1,2} (-1)^{i+1} \lambda_{i} x_1 x_2 \right)^2 + \sum_{j=4}^{d}  \sum_{i=1}^{2} \left( \mu_{i,j} x_{i} + m_{j} x_{\tau(i)}\right)^2\\
 & +\sum_{j=4}^{d} \left[ \left( \sum_{i=1}^{2} (-1)^{{i+1}}  \mu_{i,j} x_1 x_2 + m_j x_{\tau(i)}^2 \right)^2 + \left( \sum_{i=1}^{2} 
  (-1)^{i+1} (\lambda_{i} \mu_{\tau (i),j} x_1 x_2 + \lambda_{i} m_j x_i^2) \right)^{2} \right] \\
  &+ \sum_{4 \leq j < k \leq d} \left( \sum_{i=1}^{2} (-1)^{i+1} ( \mu_{i,j} x_i + m_j x_{\tau(i)})( \mu_{2,k} x_{\tau(i)}+ m_{k} x_{i}) \right)^2 + O(\rho^6).
 \end{align*}
    It is easily seen that the only sums contributing terms of order $\rho^2$ or lower are the sums on the first line.
    By expanding the squares, we get $\eqref{D(x)}$ of which $\eqref{polarD(x)}$ is a direct consequence.
  \end{proof}
  \begin{claim}\label{umbilic3}
  For $r$ small enough that $B(0,r) \subset U$, we have:
  \begin{equation}\label{2unifarea}
  \mathcal{H}^{2}(B(0,r) \cap \Omega)= \pi r^{2}+ r^{4} \int_{0}^{2\pi} \left( \frac{\overline{B}(\theta)}{8} - \frac{B(\theta)}{2} \right) d\theta + O(r^6)
  \end{equation}
   where $$B(\theta)=\sum_{l=3}^{d} B_{l}^2(\theta),$$ $$B_3 = \frac{1}{2},$$ $$B_4= \frac{\lambda_1 a_1^2+\lambda_2 a_2^2}{2},$$ and $$B_l= \frac{\mu_{1,l}a_1^2+\mu_{2,l}a_2^2+2m_la_1 a_2}{2}, \mbox{ for } l \geq 5.$$
  \end{claim}
  \begin{proof}
  Let $F: \RR^2 \to \RR^d$ be the map:
  $$F(\overline{x})=(\overline{x}, z_3(\overline{x}), \dots, z_d(\overline{x})).$$
  By the area formula, taking $r$ small enough and since $F$ is analytic, we have:
  \begin{align}
  \mathcal{H}^2(B(0,r) \cap \Omega) &= \int_{F^{-1}(B(0,r))} \sqrt{\mathcal{D}(\overline{x})}dA \\ 
  & =\int_{0}^{2\pi} \int_{0}^{\rho(\theta)} \left( 1+ \frac{\overline{B}(\theta)}{2} + O(\rho^{4}) \right) \rho d\rho d\theta , \\
  &=\int_{0}^{2\pi} \left[ \frac{\rho^{2}}{2} +\frac{\overline{B}(\theta)}{8} \rho^4 + O(\rho^6) \right]_{0}^{\rho(\theta)} d\theta . \label{areaintegral}
  \end{align}
  We now find $\rho(\theta)$.
  Note that when $x_1^2 + x_2^2 = \rho(\theta)^2$, we have $F(x_1,x_2) \in \partial B(0,r)$. Hence:
 
  \begin{equation}\label{boundary}
  \rho(\theta)^2+ \sum_{j=3}^{d} z_j^2 = r^2.
  \end{equation}
  By $\eqref{Taylor1}$, $\eqref{Taylor2}$ and $\eqref{Taylor3}$, $\eqref{boundary}$ becomes:
  \begin{align}
  \rho^{2}(\theta)+\sum_{j=3}^{d} B_j^2(\theta) \rho^4(\theta) &= r^2, \\
  \rho^{2}(\theta) + B \rho^4(\theta) &= r^2 \label{rhoandr}.
  \end{align}
 A calculation then gives:
  \begin{equation} \label{rhooftheta}
  \rho(\theta)= r - \frac{B(\theta)}{2}r^3+O(r^4),
 \end{equation}
 and consequently 
 \begin{align}
 \rho^2(\theta) &=r^2 - B(\theta)r^{4} + O(r^6), \label{rhoofthetapowers}\\
 \rho^{4}(\theta)&=r^4 + O(r^6). \nonumber
 \end{align}
 Plugging $\eqref{rhoofthetapowers}$ in $\eqref{areaintegral}$, we get:
 \begin{equation}
\mathcal{H}^{2}(B(0,r) \cap \Omega)= \pi r^{2}+ r^{4} \int_{0}^{2\pi} \left( \frac{\overline{B}(\theta)}{8} - \frac{B(\theta)}{2} \right) d\theta + O(r^6).
 \end{equation}
  \end{proof}
 Let us express $B$ in term of the $\lambda_i$'s, $\mu_{i,j}$'s and $m_j$'s.
 We have :
 \begin{align}
 B &= \sum_{l=3}^{d} B_l^2, \nonumber\\
 &=  \frac{1}{4} +\frac{1}{4}  \sum_{i=1}^{2} \left( \lambda_i^2 +  \sum_{j=5}^{d} \mu_{i,j}^{2} \right) a_{i}^4 + \frac{1}{4} \left( 2 \lambda_1 \lambda_2+ 2 \sum_{j=5}^{d} \mu_{1,j} \mu_{2,j} + 4 m_j^2 \right) a_1^2 a_2^2 \\ & + \frac{1}{4} \left( \sum_{j=5}^{d} m_j \mu_{1,j} \right) a_1^3 a_2+ \frac{1}{4} \left( \sum_{j} m_j \mu_{2,j} \right) a_1 a_2^3 \nonumber \\
 &= \frac{1}{4} \left( 1+ \delta a_1^4 + \epsilon a_2^4 + \iota a_1^2 a_2^2 + \omega a_1^3 a_2 + \kappa a_2^3 a_1\right),
 \end{align}
 
 where $$\delta=  \left( \lambda_1^2 +  \sum_{j=5}^{d} \mu_{1,j}^{2} \right),$$ $$\epsilon= \left( \lambda_2^2 +  \sum_{j=5}^{d} \mu_{2,j}^{2} \right),$$ $$\iota=\left( 2 \lambda_1 \lambda_2+ 2 \sum_{j=5}^{d} \mu_{1,j} \mu_{2,j} + 4 m_j^2 \right),$$ $$\omega= \left( \sum_{j=5}^{d} m_j \mu_{1,j} \right),$$ and $$\kappa=\left( \sum_{j} m_j \mu_{2,j} \right).$$
 We now use the fact that $\mathcal{H}^2 \res \Omega$ is  2-uniform on its support to deduce a relation between the $\lambda_i$'s, $\mu_{i,j}$'s and $m_j$'s.
\begin{claim}\label{umbilic4}
We have:
$$\lambda_1 = \lambda_2=\lambda,$$
and for all $j \geq 5$ 
$$\mu_{1,j}=\mu_{2,j}= \mu_{j} \mbox{ and } m_j=0 .$$
\end{claim}
\begin{proof}
On one hand, by Corollary $\ref{3unif}$,  we have $\mathcal{H}^2 (B(0,r) \cap \Omega)= \pi r^2$.
On the other hand, by $\eqref{2unifarea}$ , we have $\mathcal{H}^{2}(B(0,r))= \pi r^{2}+ r^{4} \int_{0}^{2\pi} \left( \frac{\overline{B}(\theta)}{8} - \frac{B(\theta)}{2} \right) d\theta + O(r^6)$.
By equating them we get 
\begin{equation}
\int_{0}^{2\pi} \frac{\overline{B}(\theta)}{8} - \frac{B(\theta)}{2} d\theta = 0.
\end{equation}
Rewrite this in term of $a_1$ and $a_2$ to get:
\begin{equation}\label{greekletters}
 \frac{\alpha}{8} \int_{0}^{2\pi} a_1^2 d\theta + \frac{\beta}{8} \int_{0}^{2\pi} a_2^2 d\theta - \frac{1}{8} \int_{0}^{2\pi} d\theta- \frac{\delta}{8} \int_{0}^{2\pi} a_1^4 d\theta - \frac{\epsilon}{8}\int_{0}^{2\pi} a_2^4 d\theta - \frac{\iota}{8} \int_{0}^{2\pi} a_1^2 a_2^2 d\theta = 0,
\end{equation}
by using the fact that 
$$ \int_{0}^{2\pi} cos(\theta) sin(\theta) d\theta = \int_{0}^{2\pi} cos^3(\theta) sin(\theta) d\theta=\int_{0}^{2\pi} cos(\theta) sin^3(\theta) d\theta=0.$$
Moreover, since 
\begin{align*}
&\int_{0}^{2\pi} cos^2(\theta)d\theta =\int_{0}^{2\pi} sin^2(\theta) d\theta= \pi \\
&\int_{0}^{2\pi} cos^{4}(\theta) d\theta =\int_{0}^{2\pi} sin^4(\theta) d\theta= \frac{3\pi}{4} \\
&\int_{0}^{2\pi} cos^2(\theta) sin^2(\theta) d\theta=\frac{\pi}{4},
\end{align*}
$\eqref{greekletters}$ becomes:
$$ 4 \alpha + 4 \beta - 8- 3 \delta - 3 \epsilon - \iota =0.
$$
Replacing the letters by their values in terms of the $\lambda_i$'s, $\mu_{i,j}$'s and $m_j$'s gives:
\begin{equation}
\left( \lambda_1^2+ \lambda_2^2 - 2 \lambda_1 \lambda_2 \right) + \sum_j \left( \mu_{1,j}^2 + \mu_{2,j}^2 - 2 \mu_{1,j} \mu_{2,j}\right) + \sum_j 4m_j^2 =0
\end{equation}
This implies that $\lambda_1 - \lambda_2 = \mu_{1,j} - \mu_{2,j}= m_j=0$ for all $j$.

\end{proof}

We can now prove Lemma $\ref{umbilic}$

\begin{proof} 
Write 
\begin{equation} 
\mathcal{R}= \cup_{i} M_i ,
\end{equation}
where each $M_i$ is a connected analytic  $2$-submanifold of $\RR^{d}$.
Since every point of $M_i$, $i>0$, is analytic, it is umbilic and therefore by Theorem $\ref{umbilicmanifold }$  
$M_i$ lies in some $2$-dimensional sphere $S_i$ (not necessarily distinct) or some $2$-plane $P_i$. The fact that the $M_i$'s are $2$-submanifolds of the unit sphere $\mathbb{S}^{d-1}$ excludes the latter case, thus ending the proof. \end{proof}

\end{document}